\documentclass[11pt,english, reqno]{amsart}

\usepackage[T1]{fontenc}
\usepackage[latin9]{luainputenc}
\usepackage{geometry}
\usepackage{amsmath}
\usepackage{amsthm}
\usepackage{amssymb}
\usepackage[active]{srcltx}
\usepackage[colorlinks]{hyperref}
\usepackage{enumitem}

 \newcommand{\eps}{{\varepsilon}}

\newcommand{\kom}[1]{}
\newcommand{\norm}[1]{\left|\left|#1\right|\right|}

\makeatletter
\numberwithin{equation}{section}
\numberwithin{figure}{section}
\theoremstyle{plain}
\newtheorem{thm}{\protect\theoremname}[section]
\newtheorem{corollary}[thm]{Corollary}
\theoremstyle{definition}
\newtheorem{defn}[thm]{\protect\definitionname}
\theoremstyle{plain}
\newtheorem{lem}[thm]{\protect\lemmaname}

\makeatother

\usepackage{babel}
\providecommand{\definitionname}{Definition}
\providecommand{\lemmaname}{Lemma}
\providecommand{\theoremname}{Theorem}

\newcommand{\abs}[1]{\left|#1\right|}

\global\long\def\d{\,d}%
\global\long\def\tr{\mathrm{tr}}%
\global\long\def\supp{\operatorname{spt}}%
\global\long\def\div{\operatorname{div}}%
\DeclareMathOperator*{\osc}{osc}
\global\long\def\l{<}%

\allowdisplaybreaks[1]

\begin{document}

\author[Lindqvist]{Peter Lindqvist}
\address{Department of Mathematical Sciences, Norwegian University of Science and Technology, NO-7491 Trondheim, Norway}
\email{peter.lindqvist@ntnu.no}

\author[Parviainen]{Mikko Parviainen}
\address{Department of Mathematics and Statistics, University of
Jyv\"askyl\"a, PO~Box~35, 40014 Jyv\"askyl\"a, Finland}
\email{mikko.j.parviainen@jyu.fi}

\author[Siltakoski]{Jarkko Siltakoski}
\address{Department of Mathematics and Statistics, University of Helsinki, P.O. Box 68, FI-00014 Helsinki, Finland}
\email{jarkko.siltakoski@helsinki.fi}

\title[Equivalence and Lipschitz]{Lipschitz continuity and equivalence of positive viscosity and weak solutions to Trudinger's equation}

\date{\today}
\thanks{MP is supported by the Research Council of Finland, project 360185. JS is supported by the Emil Aaltonen foundation. The authors would like to thank Juha Kinnunen and Erik Lindgren for useful discussions.}

\begin{abstract}
We show that stricty positive viscosity supersolutions to Trudinger's equation are local weak supersolutions.
As an application, we show using the Ishii-Lions method that not only viscosity but also weak solutions are Lipschitz continuous in the space variable. In the time variable, we obtain H{\"o}lder continuity with the exponent $1/2$.
\end{abstract}

\maketitle

\tableofcontents

\section{Introduction}

We study Trudinger's equation 
\begin{equation}
(p-1)u^{p-2}\partial_{t}u-\Delta_{p}u=0\label{eq:trudinger}
\end{equation}
in an open and bounded domain $\Xi\subset\mathbb{R}^{N+1}$, with $p>1$ and  $\Delta_{p}u=\operatorname{div}(\abs{D u}^{p-2} D u)$.
We restrict ourselves to the (strictly) positive $u$ and show that viscosity supersolutions are local weak supersolutions
(Theorem \ref{thm:visc is weak}). The reverse implication also holds so that positive weak and viscosity solutions actually coincide (Theorem~\ref{thm:equivalence}). In the past, such questions have been studied for example by Ishii \cite{Ishii95} for linear elliptic equations and by Juutinen, Lindqvist, and Manfredi for the elliptic and parabolic $p$-Laplace equation in \cite{juutinenlm01}.
As an application, we show that not only viscosity solutions but also weak solutions are Lipschitz continuous  in the space variable (Theorem~\ref{thm:lipschitz theorem} and Corollary~\ref{cor:lip}), which so far has been an open problem. We also obtain H{\"o}lder continuity with the exponent $1/2$ with respect to the time variable. Thus
\begin{align*}
    \abs{u(x,t)-u(x_0,t_0)}\le L\left( \abs{x-x_0}+\abs{t-t_0}^{1/2}\right).
\end{align*}

To prove that viscosity supersolutions are weak supersolutions, we
use the technique introduced by Julin and Juutinen \cite{juutinenj12} based on approximation by inf-convolution.  However, many problems appear since there is $u^{p-2}$ in Trudinger's equation. Problems caused by this factor in the theory are serious and, for example, full uniqueness is still an open problem. The more detailed idea of the proof is as follows.

As a first step, it holds that inf-convolution is still a
viscosity supersolution and by Alexandroff's theorem twice differentiable a.e. The details of this step are established in Lemma~\ref{lem:inf conv is super}, which is the key lemma of the paper.  For Trudinger's equation, this step is problematic since the term $u^{p-2}\partial_{t}u$ has $u$-dependence, and translation in the inf-convolution causes an extra error which will be difficult to control. Controlling the error requires careful balancing out with the gradient, and this idea is described right before the lemma. 

As a second step, we observe that the previous step implies that the $p$-Laplacian of the inf-convolution exists and  that supersolution property holds pointwise a.e.\ without difficulties at least when $p\ge 2$. 
 Then the idea is to multiply the $p$-Laplacian of the inf-convolution by a test function and integrate by parts. However, to justify this step a further regularization is required. This is taken care of in Lemma~\ref{lem:inf conv is weak}, and again the error term will have to be taken into account.

Finally, as a third step, one needs to be able to pass to a limit under the integral.  Details of this step are in Section~\ref{sec:energy-estimates}. To be able to pass to a limit under the integral sign, we need a uniform bound for the gradients of uniformly bounded weak supersolutions with the error. This is established in Lemma~\ref{lem:caccioppoli}. Also observe that the weak convergence is not sufficient here. In order to pass to a limit under the integral sign, we need stronger convergence for the gradients obtained in Lemma~\ref{lem:strong conv}.

 Using the above procedure, we finally obtain in Theorem~\ref{thm:visc is weak} that a viscosity supersolution is a weak supersolution under our assumptions. We also establish the results in the case $1<p<2$ using a special inf-convolution. Finally, the standard argument using Theorem 3 in \cite{lindgrenl22} completes the picture and gives that weak solutions are viscosity solutions, and thus the equivalence result in Theorem~\ref{thm:equivalence}.

As an application of the preceding equivalence result, we establish a qualitative Lipschitz result for positive viscosity and weak solutions  in Section \ref{eq:Ishii-Lions lemma cnd}, see Theorem~\ref{thm:lipschitz theorem}. In order to obtain this, we use the  Ishii-Lions technique  \cite{ishiiLions90} (see also \cite{imbertJinSilvestre19}) from the theory of viscosity solutions twice. Since we showed above that viscosity solutions are weak solutions, this also implies Lipschitz continuity of positive weak solutions.

The equation (\ref{eq:trudinger}) was originally suggested by Trudinger in \cite{trudinger68} as an example of the equation that could have simpler Harnack's inequality than the standard parabolic $p$-Laplace equation. 
Despite the homogeneity of the equation, the regularity theory turned out to be quite problematic.  In \cite{kinnunenk07}, Kinnunen and Kuusi proved Harnack's inequality for Trudinger's equation using Moser's iteration, whereas Gianazza and Vespri established Harnack's inequality in \cite{gianazzav06} using De Giorgi's method. In \cite{kuusisu12} Kuusi, Siljander and Urbano obtained  H{\"o}lder continuity  in the case $2<p$ for nonnegative solutions, and in \cite{kuusilsu12} the case $1<p<2$ was treated by Kuusi, Laleoglu, Siljander and Urbano.  In \cite{diehlu20} Diehl and Urbano established a sharp H{\"o}lder regularity for Trudinger's inhomogeneous equation.  Recently, the asymptotic behaviour and comparison principle for positive solutions  were studied by Hynd and Lindgren \cite{hyndl21}  as well as Lindgren and Lindqvist \cite{lindgrenl22} respectively.

Also the doubly nonlinear equation,
\begin{align*}
\partial_{t}(|u|^{q-1}u)-\Delta_{p}u=0
\end{align*}
has received a lot of attention. In \cite{henriques20}, Henriques established H\"older regularity for nonnegative weak solutions. 
B{\"o}gelein, Duzaar, Liao and Sch{\"a}tzler  studied the H\"older regularity results for the doubly nonlinear equation in a series of papers \cite{bogeleindl21, bogeleindls22, liaos22}, where the first paper covers Trudinger's equation also  with sign-changing solutions. Misawa \cite{misawa23} considered the expansion of positivity for doubly nonlinear parabolic equations, including Trudinger's equation. In a recent work \cite{bogeleindgls}, B\"ogelein, Duzaar, Gianazza, Liao and Scheven consider continuity of the gradient of solutions to doubly non-linear equations in a certain range of parameters. However, this does not cover Trudinger's equation:  \emph{To the best of our knowledge, there is no preceding correct Lipschitz proof in the literature}. From our theorem, it follows that the gradient of a solution exists a.e.\ by Rademacher's theorem in the usual sense in addition to the weak sense, and also that the gradient is \emph{locally bounded}.    

As for an alternative point of view, Bhattacharya and Marazzi studied  viscosity solutions to Trudinger's equation and their asymptotic behaviour in \cite{bhattacharyam15} and \cite{bhattacharyam16}. They proved existence and  uniqueness for positive solutions. However, their proof is based on a logarithmic change of variables and is quite different from the techniques of this paper. As mentioned, we follow the approach in \cite{juutinenj12}.  The technique was extended to the parabolic case in \cite[Appendix A]{parviainenv20} and \cite{siltakoski21}. This approach does not rely on the uniqueness theory of viscosity solution, which could also be used to study the equivalence of weak and viscosity solutions as for example in \cite{juutinenlm01}.   
The aforementioned article \cite{lindgrenl22} also briefly discusses viscosity solutions and establishes that nonnegative weak solutions are viscosity solutions. Moreover, in \cite{hyndl21} Hynd and Lindgren study the asymptotic behaviour and show that a related quantity converges to an extremal of Poincar\'e's inequality.
To conclude, we note that Hynd and Lindgren investigated the Lipschitz regularity of viscosity solutions for a different class of doubly nonlinear parabolic equations in \cite{hyndl19}.

\subsection{Viscosity solutions}

We consider positive, continuous viscosity sub- and
supersolutions.
\begin{defn}[Viscosity solutions]
Let $u:\Xi\rightarrow(0,\infty)$ be continuous. We say that $u$
is a \textit{viscosity supersolution} to (\ref{eq:trudinger}) in
$\Xi$ if whenever $\phi\in C^{2}(\Xi)$ touches $u$ from below
at $(x,t)\in\Xi$, $D\phi(y,s)\neq 0, y\neq x$, we have
\[
\limsup_{(y,s)\rightarrow(x,t), y\not=x}\left(\partial_{t}\big(\abs{\phi(y,s)}^{p-2}\phi(y,s)\big)-\Delta_{p}\phi(y,s)\right)\geq 0.
\]
Analogously, we say that $u$ is a \textit{viscosity subsolution}
to (\ref{eq:trudinger}) in $\Xi$ if whenever $\varphi\in C^{2}(\Xi)$
touches $u$ from above at $(x,t)\in\Xi$, $D\phi(y,s)\neq 0, y\neq x$, we have
\[
\liminf_{(y,s)\rightarrow(x,t), y\not = x}\left(\partial_{t}\big(\abs{\phi(y,s)}^{p-2}\phi(y,s)\big)-\Delta_{p}\phi(y,s)\right)\leq0.
\]
Finally, $u$ is a \textit{viscosity solution} if it is both viscosity
sub- and supersolution.
\end{defn}
Since we consider positive solutions, the above definition takes the form 
\[
\limsup_{(y,s)\rightarrow(x,t), y\not=x}\left((p-1)\partial_{t}\phi(y,s)-u^{2-p}(y,s)\Delta_{p}\phi(y,s)\right)\geq 0.
\]
for supersolutions, and analogously for subsolutions. 
In the case $p\ge 2$, the definition of the viscosity solutions to Trudinger's equation in terms of parabolic jets is in \cite{bhattacharyam15}; here we allow the whole range $1<p<\infty$.
When using Theorem on Sums, we also need the notion of parabolic jets. For example we denote the parabolic subjet as $(\theta,\eta,X)\in\mathcal{P}^{2,-}u(x,t)$. For more details, see  \cite{crandall97, userguide, lindqvist25, diehlFrizOberhauser14}.

Weak solutions to (\ref{eq:trudinger}) are defined below, where $\d z=\d x \d t$.
\begin{defn}[Weak solutions]
 Let $u\in L_{\text{loc}}^{p}(\Xi)$ with $Du\in L_{\text{loc}}^{p}(\Xi)$,
where $Du$ is the distributional gradient of $u$ in space variable.
We say that $u$ is a \textit{weak supersolution} to (\ref{eq:trudinger})
if for all non-negative $\varphi\in C_{0}^{\infty}(\Xi)$ we have
\[
\int_{\Xi}\Big(-{u}^{p-2}u\,\partial_{t}\varphi+\left|Du\right|^{p-2}Du\cdot D\varphi\big)\d z\geq 0.
\]
For \textit{weak subsolutions}, the above inequality is reversed and
$u$ is a \textit{weak solution} if it is both weak sub- and supersolution.
\end{defn}

\section{Positive, uniformly continuous viscosity supersolutions are weak
supersolutions}

\subsection{Regularization via inf-convolution}
\label{sec:inf-conv}
We use a special form of inf-convolution where the penalty
term in time has different weight than the penalty term in space.
Given $u:\Xi\rightarrow\mathbb{R}$, we denote 
\begin{equation}
u_{\varepsilon}(x,t):=\inf_{(y,s)\in\Xi}\left\{ u(y,s)+\frac{\left|y-x\right|^{q}}{q\varepsilon^{q-1}}+\frac{\left|s-t\right|^{2}}{2\delta_{\varepsilon}}\right\} ,\label{eq:inf conv special}
\end{equation}
where $\delta_{\varepsilon}$ is a suitable function of $\varepsilon$,
$q>p/(p-1)$ if $1<p<2$, and $q=2$ if $p\geq2$. By choosing small
$\delta_{\varepsilon}$ compared to $\varepsilon$, we may ensure
that the time penalty term is small compared to $\varepsilon$. This
allows us to control an error term in the inequality
satisfied by the inf-convolution of a supersolution. 

The following lemma collects some well known properties of inf-convolution. We postpone its proof into the appendix.
\begin{lem}
\label{lem:inf properties}Assume that $u:\Xi\rightarrow\mathbb{R}$ is lower semicontinuous
and bounded. Suppose also that $\delta_{\varepsilon}\rightarrow0$
as $\varepsilon\rightarrow0$. Then $u_{\varepsilon}$ has the following
properties.
\begin{enumerate}[label=(\roman*)]
\item We have $u_{\varepsilon}\leq u$ in $\Xi$ and $u_{\varepsilon}\rightarrow u$
pointwise as $\varepsilon\rightarrow0$.
\item Denote $r(\varepsilon):=(q\varepsilon^{q-1}\osc_{\Xi}u)^{\frac{1}{q}}$,
$t(\varepsilon):=(2\delta_{\varepsilon}\osc_{\Xi}u)^{\frac{1}{2}}$ and set
\[
\Xi_{\varepsilon}:=\left\{ (x,t)\in\Xi:B_{r(\varepsilon)}(x)\times(t-t(\varepsilon),t+t(\varepsilon))\Subset\Xi\right\} .
\]
Then, for any $(x,t)\in\Xi_{\varepsilon}$ there exists $(x_{\varepsilon},t_{\varepsilon})\in\overline{B}_{r(\varepsilon)}(x)\times[t-t(\varepsilon),t+t(\varepsilon)]$
such that
\[
u_{\varepsilon}(x,t)=u(x_{\varepsilon},t_{\varepsilon})+\frac{\left|x-x_{\varepsilon}\right|^{q}}{q\varepsilon^{q-1}}+\frac{\left|t-t_{\varepsilon}\right|^{2}}{2\delta_{\varepsilon}}.
\]
\item The function $u_{\varepsilon}$ is semi-concave in $\Xi_{\varepsilon}$.
In particular, the function $u_{\varepsilon}(x,t)-(C\left|x\right|^{2}+t^{2}/\delta_{\varepsilon})$
is concave in $\Xi_{\varepsilon}$, where $C=(q-1)r(\varepsilon)^{q-2}/\varepsilon^{q-1}$.
\item Suppose that $u_{\varepsilon}$ is twice differentiable in space and
time at $(x,t)\in\Xi_{\varepsilon}$. Then 
\begin{align*}
\partial_{t}u_{\varepsilon}(x,t) & =\frac{t-t_{\varepsilon}}{\delta_{\varepsilon}},\\
Du_{\varepsilon}(x,t) & =(x-x_{\varepsilon})\frac{\left|x-x_{\varepsilon}\right|^{q-2}}{\varepsilon^{q-1}},\\
D^{2}u_{\varepsilon}(x,t) & \leq(q-1)\frac{\left|x-x_{\varepsilon}\right|^{q-2}}{\varepsilon^{q-1}}I.
\end{align*}
\end{enumerate}
\end{lem}

The following lemma gives an estimate for the penalty term
in inf-convolution.
\begin{lem}
\label{lem:inf conv error est} Suppose that $u:\Xi\rightarrow\mathbb{R}$
is uniformly continuous. Let $\omega_{t}$ and $\omega_{x}$ be bounded
moduli of continuity of $u$ in time and space, respectively. Suppose
moreover that $\omega_{t}$ is strictly increasing. Let $u_{\varepsilon}$
be defined by (\ref{eq:inf conv special}), where 
\begin{equation}\label{eq:delta_epsilon}
\delta_{\varepsilon}:=\frac{(\omega_{t}^{-1}(\varepsilon^{q-1}))^{2}}{2\osc_{\Xi}u}.
\end{equation}
Let $(x_{\varepsilon},t_{\varepsilon})$ be such that 
\begin{align*}
 u_{\varepsilon}(x,t)=u(x_{\varepsilon},t_{\varepsilon})+\frac{\left|x-x_{\varepsilon}\right|^{q}}{q\varepsilon^{q-1}}+\frac{\left|t_{\varepsilon}-t\right|^{2}}{2\delta_{\varepsilon}}.   
\end{align*}
Then
\[
\frac{\left|t-t_{\varepsilon}\right|^{2}}{2\delta_{\varepsilon}}\leq\varepsilon^{q-1}\quad\text{and}\quad\left|x-x_{\varepsilon}\right|\leq q^{1/q}\varepsilon^{\frac{q-1}{q}}\omega_{x}^{1/q}(\left|x-x_{\varepsilon}\right|).
\]
\end{lem}

\begin{proof}
By the definition of inf-convolution, we have
\begin{align*}
 & \frac{\left|t-t_{\varepsilon}\right|^{2}}{2\delta_{\varepsilon}}\le u(x_{\varepsilon},t)-u(x_{\varepsilon},t_{\varepsilon}),\\
 & \frac{\left|x-x_{\varepsilon}\right|^{q}}{q\varepsilon^{q-1}}\leq u(x,t_{\varepsilon})-u(x_{\varepsilon},t_{\varepsilon}).
\end{align*}
From the first inequality, we estimate $\left|t-t_{\varepsilon}\right|\leq\sqrt{2\delta_{\varepsilon}\osc_{\Xi}u}$.
Using the first inequality again and the definition of $\delta_{\varepsilon}$,
we obtain
\[
\frac{\left|t-t_{\varepsilon}\right|^{2}}{2\delta_{\varepsilon}}\leq\omega_{t}(\left|t-t_{\varepsilon}\right|)\leq\omega_{t}\left(\sqrt{2\delta_{\varepsilon}\osc_{\Xi}u}\right)\leq\omega_{t}\left(\omega_{t}^{-1}(\varepsilon^{q-1})\right)=\varepsilon^{q-1}.
\]
Similarly, from the second inequality we obtain that $\left|x-x_{\varepsilon}\right|\leq(q\varepsilon^{q-1}\omega_{x}(\left|x-x_{\varepsilon}\right|))^{1/q}$.
\end{proof}
The next elementary inequality immediately follows by integrating the derivative of $h\mapsto h^{\frac{2-p}{2}}$. For the convenience of the reader this is written down in the appendix.
\begin{lem}[Elementary inequality]
\label{lem: elem ineq} Let $p>1$. Suppose that $a\geq m>0$. Then
for all $s>0$ we have
\[
\left|(a+s)^{\frac{2-p}{2}}-a^{\frac{2-p}{2}}\right|\leq C(p,m)s.
\]
\end{lem}

Next we prove a key lemma: the inf-convolution of a positive, uniformly
continuous viscosity supersolution to Trudinger's equation solves
the same equation with \emph{an additional first-order error term}. Let us
briefly highlight key arguments of the proof before the full details.
Observe that formally, a positive solution $u$ to Trudinger's equation
solves
\[
(p-1)\partial_{t}u-u^{2-p}\Delta_{p}u\geq0.
\]
This form is useful for estimating the error term generated by the
inf-convolution $u_{\varepsilon}$, as the error generating $u$-dependence
is now decoupled from the time derivative to the elliptic term and
the Jensen-Ishii lemma can be used to control the behavior of the
elliptic term. Using also the assumption that $0<m\leq u\leq M$,
the argument in \cite{Ishii95} and for the $p$-Laplacian in \cite{siltakoski18} readily
gives that $u_{\varepsilon}$ is a viscosity supersolution to
\begin{equation}
(p-1)u_{\varepsilon}^{p-2}\partial_{t}u_{\varepsilon}-\Delta_{p}u_{\varepsilon}\geq-\mathcal{E}\quad\text{in}\quad\Xi_{\varepsilon},\label{eq:blabla}
\end{equation}
where $\Xi_{\varepsilon}$ is an in Lemma \ref{lem:inf properties}
and the error term has the form (assume for simplicity here that $p\geq2$
so that $q=2$)
\[
\mathcal{E}=C\frac{\left|Du_{\varepsilon}\right|^{p-2}}{\varepsilon}\left(\frac{\left|x-x_{\varepsilon}\right|^{4}}{\varepsilon^{2}}+\frac{\left|t-t_{\varepsilon}\right|^{4}}{\delta_{\varepsilon}^{2}}\right).
\]
Here, roughly speaking, the term $\left|Du_{\varepsilon}\right|^{p-2}/\varepsilon$
is generated by $\Delta_{p}u$ and the term in brackets comes from
$u^{2-p}$. The term with $\left|t-t_{\varepsilon}\right|$ can be
estimated via Lemma \ref{lem:inf conv error est} and the choice of
$\delta_{\varepsilon}$. The term with $\left|x-x_{\varepsilon}\right|^{4}/\varepsilon^{3}$
is trickier to deal with, as this ratio may blow up as $\varepsilon\rightarrow0$
unless $u$ is known to be, say, Lipschitz. To deal with the problem, we
use the fact that $Du_{\varepsilon}=\frac{x-x_{\varepsilon}}{\varepsilon}$.
This together with the uniform continuity of $u$ lets us absorb some
of the problematic ratio to the gradient part in (\ref{eq:error est 2}). This way, we arrive
at the estimate
\[
\mathcal{E}\leq E_{\varepsilon}\left(\left|Du_{\varepsilon}\right|^{p}h(\left|Du_{\varepsilon}\right|)+\left|Du_{\varepsilon}\right|^{p-2}\right),
\]
where $h:[0,\infty)\rightarrow[0,\infty)$ is continuous (and thus bounded at $0$) and $h(s)\rightarrow0$
as $s\rightarrow \infty$ and $E_{\varepsilon}$ is a number that converges
to zero as $\varepsilon\rightarrow0$. The term $h(\left|Du_{\varepsilon}\right|)$
is convenient in (\ref{eq:energy est 5}), as its presence lets us prove an energy estimate
that ensures the weak convergence of a subsequence of $Du_{\varepsilon}$ in $L_{loc}^{p}$.
\begin{lem}
\label{lem:inf conv is super} Let $u$ be a uniformly continuous
viscosity supersolution to 
\[
(p-1)u^{p-2}\partial_{t}u-\Delta_{p}u \ge 0\quad\text{in}\quad\Xi.
\]
Suppose that $0<m\leq u\leq M$ in $\Xi$ and let $u_{\varepsilon}$
be defined by (\ref{eq:inf conv special}), where $\delta_\varepsilon$ is as in (\ref{eq:delta_epsilon}). Then, if $(\theta,\eta,X)\in\mathcal{P}^{2,-}u_{\varepsilon}(x,t)$
with $\eta\not=0$, and $(x,t)\in\Xi_{\varepsilon}$, we have
\[
(p-1)u_{\varepsilon}^{p-2}\theta-\left|\eta\right|^{p-2}\tr\left((I+(p-2)\frac{\eta\otimes\eta}{\left|\eta\right|^{2}})X\right)\geq-E_{\varepsilon}\left(\left|\eta\right|^{p}h(\left|\eta\right|)+\left|\eta\right|^{\max(p-2,0)}\right),
\]
where $E_{\varepsilon}\geq0$ is a number that converges to zero as
$\varepsilon\rightarrow0$, $h:[0,\infty)\rightarrow\mathbb{R}$ is
continuous, strictly decreasing and $h(s)\rightarrow0$ as $s\rightarrow\infty$. 

Furthermore, in the case that $(\theta,0,X)\in\mathcal{P}^{2,-}u_{\varepsilon}(x,t)$,
we have $\theta\geq0$. 
\end{lem}
\begin{proof}
Let us first consider the case $\eta\neq 0$.
Since $(x,t)\in\Xi_{\varepsilon}$, by Lemma \ref{lem:inf properties},
there is $(x_{\varepsilon},t_{\varepsilon})$ such that 
\begin{equation}\label{eq:inf lemma 21}
u_{\varepsilon}(x,t)=u(x_{\varepsilon},t_{\varepsilon})+\frac{\left|x-x_{\varepsilon}\right|^{q}}{q\varepsilon^{q-1}}+\frac{\left|t-t_{\varepsilon}\right|^{2}}{2\delta_{\varepsilon}}=:u(x_{\varepsilon},t_{\varepsilon})+A_\varepsilon,
\end{equation}
where $\delta_\varepsilon$ is defined by (\ref{eq:delta_epsilon}).
It was assumed that $(\theta, \eta,X)\in\mathcal{P}^{2,-}u_{\varepsilon}(x,t)$, so there exists
$\varphi\in C^{2}$ such that it touches $u_{\varepsilon}$ from below
at $(x,t)$ and $(\partial_{t}\varphi(x,t),D\varphi(x,t),D^{2}\varphi(x,t))=(\theta,\eta,X)$.
By direct computation
\[
\theta=\frac{t-t_{\varepsilon}}{\delta_{\varepsilon}}\quad\text{and}\quad\eta=\frac{x-x_{\varepsilon}}{\varepsilon^{q-1}}\left|x-x_{\varepsilon}\right|^{q-2}.
\]
Using the definition of inf-convolution and that $\varphi$ touches
$u_{\varepsilon}$ from below at $(x,t)$, we have for all $(y,s),(z,\tau)\in\Xi$
that
\begin{align}
-u(y,s)+\varphi(z,\tau)-\frac{\left|y-z\right|^{q}}{q\varepsilon^{q-1}}-\frac{\left|s-\tau\right|^{2}}{2\delta_{\varepsilon}} & \leq-u_{\varepsilon}(z,\tau)+\varphi(z,\tau)\nonumber \\
 & \leq 0\label{eq:inf conv is super 1} \\
 &=-u_{\varepsilon}(x,t)+\varphi(x,t) \nonumber\\
 & =-u(x_{\varepsilon},t_{\varepsilon})+\varphi(x,t)-\frac{\left|x-x_{\varepsilon}\right|^{q}}{q\varepsilon^{q-1}}-\frac{\left|t_{\varepsilon}-t\right|^{2}}{2\delta_{\varepsilon}}.\nonumber
\end{align}
That is, if we denote
\[
\Phi(y,s,z,\tau):=\frac{\left|y-z\right|^{q}}{q\varepsilon^{q-1}}+\frac{\left|s-\tau\right|^{2}}{2\delta_{\varepsilon}},
\]
the function $(y,s,z,\tau)\mapsto-u(y,s)+\varphi(z,\tau)-\Phi(y,s,z,\tau)$
has maximum at the point $(y,s,z,\tau)=(x_{\varepsilon},t_{\varepsilon},x,t)=:a$. 

Invoking the Jensen-Ishii lemma (Theorem on Sums), we conjure symmetric $N\times N$
matrices $Z$ and $Y$ such that
\begin{align*}
(\partial_{s}\Phi(a),D_{y}\Phi(a),Y) & \in\overline{\mathcal{P}}^{2,+}(-u(x_{\varepsilon},t_{\varepsilon})),\\
(-\partial_{\tau}\Phi(a),-D_{z}\Phi(a),Z) & \in\overline{\mathcal{P}}^{2,-}(-\varphi(x,t)).
\end{align*}
Since $\partial_{s}\Phi(a)=\frac{t_{\varepsilon}-t}{\delta_{\varepsilon}}=-\theta=-\partial_{\tau}\Phi(a)$
and $D_{y}\Phi(a)=\frac{x_{\varepsilon}-x}{\varepsilon^{q-1}}\left|x-x_{\varepsilon}\right|^{q-2}=-\eta=-D_{z}\Phi(x,x_{\varepsilon})$,
these can be written as
\[
(\theta,\eta,-Y)\in\overline{\mathcal{P}}^{2,-}u(x_{\varepsilon},t_{\varepsilon})\quad\text{and}\quad(\theta,\eta,-Z)\in\overline{\mathcal{P}}^{2,+}\varphi(x,t).
\]
Moreover, we have
\begin{align}
\begin{pmatrix}Y & 0\\
0 & -Z
\end{pmatrix} & \leq D_{(y,z)}^{2}\Phi(x_{\varepsilon},x)+\varepsilon^{q-1}(D_{(y,z)}^{2}\Phi(x_{\varepsilon},x))^{2}\nonumber \\
 & =\begin{pmatrix}B & -B\\
-B & B
\end{pmatrix}+2\varepsilon^{q-1}\begin{pmatrix}B^{2} & -B^{2}\\
-B^{2} & B^{2}
\end{pmatrix},\label{eq:matrix ineq}
\end{align}
where 
\[
B=\frac{1}{\varepsilon^{q-1}}\left|x-x_{\varepsilon}\right|^{q-4}((q-2)(x_{\varepsilon}-x)\otimes(x_{\varepsilon}-x)+\left|x_{\varepsilon}-x\right|^{2}I).
\]
This implies  $Y\leq Z\leq-D^{2}\varphi(x,t)=-X$. Denote
below
\[
F(\eta,X):=-\left|\eta\right|^{p-2}(\tr X+(p-2)\frac{\eta^{\prime}X\eta}{\left|\eta\right|^{2}}),
\]
where $\eta^{\prime}$ denotes the transpose of the column vector $\eta$.
Since $u$ is a viscosity supersolution, we have by the degenerate ellipticity
of $F$ that
\begin{align*}
 & (p-1)\theta+\frac{1}{u_{\varepsilon}^{p-2}(x,t)}F(\eta,X)\\
 & =(p-1)\theta+\frac{\left|\eta\right|^{p-2}}{u^{p-2}(x_{\varepsilon},t_{\varepsilon})}F(\eta,-Y)-\frac{\left|\eta\right|^{p-2}}{u^{p-2}(x_{\varepsilon},t_{\varepsilon})}F(\eta,-Y)+\frac{\left|\eta\right|^{p-2}}{u_{\varepsilon}^{p-2}(x,t)}F(\eta,X)\\
 & \geq-\frac{\left|\eta\right|^{p-2}}{u^{p-2}(x_{\varepsilon},t_{\varepsilon})}F(\eta,-Y)+\frac{\left|\eta\right|^{p-2}}{u_{\varepsilon}^{p-2}(x,t)}F(\eta,-Z)\\
 & =\frac{\left|\eta\right|^{p-2}}{u_{\varepsilon}^{p-2}(x,t)}\tr((I+(p-2)\frac{\eta\otimes\eta}{\left|\eta\right|^{2}})Z)-\frac{\left|\eta\right|^{p-2}}{u^{p-2}(x_{\varepsilon},t_{\varepsilon})}\tr((I+(p-2)\frac{\eta\otimes\eta}{\left|\eta\right|^{2}})Y).
\end{align*}
Assuming that $\varepsilon$ is small enough, we may suppose that
$\left|x-x_{\varepsilon}\right|\leq1$. We multiply the inequality
(\ref{eq:matrix ineq}) by the vector $(u^{\frac{2-p}{2}}(x_{\varepsilon},t_{\varepsilon})\frac{\eta'}{\left|\eta\right|}, u_{\varepsilon}^{\frac{2-p}{2}}(x,t)\frac{\eta'}{\left|\eta\right|})$ from left and right,
and use the elementary inequality in Lemma \ref{lem: elem ineq} to
obtain
\begin{align*}
u_{\varepsilon}^{2-p}(x,t)\frac{\eta^{\prime}Z\eta}{\left|\eta\right|^{2}}-u^{2-p}(x_{\varepsilon},t_{\varepsilon})\frac{\eta^{\prime}Y\eta}{\left|\eta\right|^{2}} & \geq-\left(u_{\varepsilon}^{\frac{2-p}{2}}(x,t)-u^{\frac{2-p}{2}}(x_{\varepsilon},t_{\varepsilon})\right)^{2}\frac{\eta^{\prime}}{\left|\eta\right|}\left(B+2\varepsilon^{q-1}B^{2}\right)\frac{\eta^{\prime}}{\left|\eta\right|}\\
 & \geq-\left(u_{\varepsilon}^{\frac{2-p}{2}}(x,t)-u^{\frac{2-p}{2}}(x_{\varepsilon},t_{\varepsilon})\right)^{2}(\left\Vert B\right\Vert +2\varepsilon^{q-1}\left\Vert B^{2}\right\Vert )\\
 & \geq-\frac{C(q)}{\varepsilon^{q-1}}\left|x-x_{\varepsilon}\right|^{q-2}\left(u_{\varepsilon}^{\frac{2-p}{2}}(x,t)-u^{\frac{2-p}{2}}(x_{\varepsilon},t_{\varepsilon})\right)^{2}\\
 & =-\frac{C(q)}{\varepsilon^{q-1}}\left|x-x_{\varepsilon}\right|^{q-2}\left((u(x_{\varepsilon},t_{\varepsilon})+A_\varepsilon)^{\frac{2-p}{2}}-u^{\frac{2-p}{2}}(x_{\varepsilon},t_{\varepsilon})\right)^{2}\\
 & \geq-C(m,p,q)\frac{\left|x-x_{\varepsilon}\right|^{q-2}}{\varepsilon^{q-1}}A^{2}_\varepsilon,
\end{align*}
where the inequality $\left|x-x_{\varepsilon}\right|\leq1$ was used
in the estimate of $\left\Vert B^{2}\right\Vert $, $\norm{B}=\max\{\abs{\lambda}\,:\,\lambda \text{ is an eigenvalue of }B\}$, and $A_\varepsilon$ was defined in (\ref{eq:inf lemma 21}). Similarly, multiplying
the matrix inequality (\ref{eq:matrix ineq}) by the vectors $(u_{\varepsilon}^{2-p}(x,t)e_{i},u^{2-p}(x_{\varepsilon},t_{\varepsilon})e_{i})$,
$i\in\left\{ 1,\ldots,N\right\} $, we see that
\[
u_{\varepsilon}^{2-p}(x,t)\tr(Z)-u^{2-p}(x_{\varepsilon},t_{\varepsilon})\tr(Y)\geq-C(N,m,p,q)\frac{\left|x-x_{\varepsilon}\right|^{q-2}}{\varepsilon^{q-1}}A^{2}_\varepsilon.
\]
Thus, combining the last three displays, we arrive at
\begin{equation}
(p-1)\theta+\frac{1}{u_{\varepsilon}^{p-2}(x,t)}F(\eta,X)\geq-C(N,p,q,m)\left|\eta\right|^{p-2}\frac{\left|x-x_{\varepsilon}\right|^{q-2}}{\varepsilon^{q-1}}A^{2}_\varepsilon.\label{eq:inf conv is super 3}
\end{equation}
Now we proceed to estimate the error term. Let us consider the cases $1<p<2$ and $p\geq2$ separately.

\textbf{Case $p\geq2$:} We use the definition of $A_\varepsilon$, the choice
$q=2$, the inequality $(c+d)^{2}\leq 2(c^{2}+d^{2})$, and the estimate
$\frac{\left|t-t_{\varepsilon}\right|^{2}}{\delta_{\varepsilon}}\leq 2 \varepsilon$
from Lemma \ref{lem:inf conv error est} to obtain
\begin{align}
\frac{A^{2}_\eps}{\varepsilon}\leq\frac{C}{\varepsilon}(\frac{\left|x-x_{\varepsilon}\right|^{4}}{\varepsilon^{2}}+\frac{\left|t-t_{\varepsilon}\right|^{4}}{\delta_{\varepsilon}^{2}}) & \leq C\frac{\left|x-x_{\varepsilon}\right|^{4}}{\varepsilon^{3}}+C\varepsilon.\label{eq:error est 1}
\end{align}
Next, recall that by Lemma \ref{lem:inf conv error est} we have
\begin{equation}
\left|x-x_{\varepsilon}\right|\leq2^{1/2}\omega_{x}^{1/2}(\left|x-x_{\varepsilon}\right|)\varepsilon^{1/2}\leq C_{0}\varepsilon^{1/2}\label{eq:error est 3}
\end{equation}
We define the function $h:[0,\infty)\rightarrow[0,\infty)$ by $h(s):=\omega_{x}^{1/2}(C_{0}^{2}s^{-1})$
when $s\in(0,\infty)$ and $h(0):=\lim_{s\rightarrow0}h(s)$. Observe
that since $u$ is bounded, we may take $\omega_x$ to be constant for large values. Therefore, $h$ is well defined, continuous, strictly decreasing and $h(s)\rightarrow0$
as $s\rightarrow\infty$. Using that $\left|\eta\right|=\frac{\left|x-x_{\varepsilon}\right|}{\varepsilon}$,
we estimate using both inequalities in (\ref{eq:error est 3})
\begin{align}
\frac{\left|x-x_{\varepsilon}\right|^{4}}{\varepsilon^{3}}=\left|\eta\right|^{2}h(\left|\eta\right|)\frac{\left|x-x_{\varepsilon}\right|^{2}}{\varepsilon}\cdot\frac{1}{h(\frac{\left|x-x_{\varepsilon}\right|}{\varepsilon})} & \leq2\left|\eta\right|^{2}h(\left|\eta\right|)\frac{\omega_{x}(\left|x-x_{\varepsilon}\right|)}{h(\frac{C_{0}}{\varepsilon^{1/2}})}\nonumber \\
 & \leq2\left|\eta\right|^{2}h(\left|\eta\right|)\frac{\omega_{x}(C_{0}\varepsilon^{1/2})}{h(\frac{C_{0}}{\varepsilon^{1/2}})}\nonumber \\
 & =2\left|\eta\right|^{2}h(\left|\eta\right|)\frac{\omega_{x}(C_{0}\varepsilon^{1/2})}{\omega_{x}^{1/2}(C_{0}^{2}\frac{\varepsilon^{1/2}}{C_{0}})}\nonumber \\
 & =2\left|\eta\right|^{2}h(\left|\eta\right|)\omega_{x}^{1/2}(C_{0}\varepsilon^{1/2}).\label{eq:error est 2}
\end{align}
Combining (\ref{eq:inf conv is super 3}), (\ref{eq:error est 1})
and (\ref{eq:error est 2}), we have
\[
(p-1)\theta-\frac{1}{u_{\varepsilon}^{p-2}(x,t)}F(\eta,X)\geq-E_{\varepsilon}(\left|\eta\right|^{p}h(\left|\eta\right|)+\left|\eta\right|^{p-2}),
\]
where we denoted
\begin{equation*}
    E_{\varepsilon}:=C(N,p,m,M)\max(\omega_{x}^{1/2}(C_{0}\varepsilon^{1/2}),\varepsilon).
\end{equation*}
This yields the desired estimate in the case $p\geq2$ since $C^{-1}(p,m,M)\leq u_{\varepsilon}^{p-2}\leq C(p,m,M)$.

\textbf{Case $1<p<2$:} Now $q\not=2$, so we have
\begin{align*}
\left|\eta\right|^{p-2}\frac{\left|x-x_{\varepsilon}\right|^{q-2}}{\varepsilon^{q-1}}A^{2} & \leq C\left|\eta\right|^{p-2}\frac{\left|x-x_{\varepsilon}\right|^{q-2}}{\varepsilon^{q-1}}\left(\frac{\left|x-x_{\varepsilon}\right|^{2q}}{q^{2}\varepsilon^{2(q-1)}}+\frac{\left|t-t_{\varepsilon}\right|^{4}}{4\delta_{\varepsilon}^{2}}\right)\\
 & =C(q)\left|\eta\right|^{p-2}\frac{\left|x-x_{\varepsilon}\right|^{3q-2}}{q^{2}\varepsilon^{3(q-1)}}+\left|\eta\right|^{p-2}\frac{\left|x-x_{\varepsilon}\right|^{q-2}}{\varepsilon^{q-1}}\frac{\left|t-t_{\varepsilon}\right|^{4}}{4\delta_{\varepsilon}^{2}}.\\
 & =:T_{1}+T_{2}.
\end{align*}
Recall that by Lemma \ref{lem:inf conv error est}
\begin{equation}
\left|x-x_{\varepsilon}\right|\leq\omega_{x}^{1/q}(\left|x-x_{\varepsilon}\right|)\varepsilon^{\frac{q-1}{q}}\leq C_{0}(q,M)\varepsilon^{\frac{q-1}{q}}.\label{eq:error est 4}
\end{equation}
To estimate $T_{2}$ we use that $q>p/(p-1)$ implies $(q-1)(p-1)-1>0$.
We also use the estimate $\frac{\left|t-t_{\varepsilon}\right|^{2}}{\delta_{\varepsilon}}\leq\varepsilon^{q-1}$
from Lemma \ref{lem:inf conv error est}, and inequality (\ref{eq:error est 4}).
We obtain 
\begin{align*}
T_{2}=\frac{\left|\eta\right|^{p-2}\left|x-x_{\varepsilon}\right|^{q-2}}{4\varepsilon^{q-1}}\frac{\left|t-t_{\varepsilon}\right|^{4}}{\delta_{\varepsilon}^{2}} & =\frac{\left|x-x_{\varepsilon}\right|^{(q-1)(p-1)-1}}{4\varepsilon^{q-1}}\frac{\left|t-t_{\varepsilon}\right|^{4}}{\delta_{\varepsilon}^{2}}\\
 & \leq C(q,M)\frac{\varepsilon^{\frac{q-1}{q}((q-1)(p-1)-1)}\varepsilon^{2(q-1)}}{\varepsilon^{q-1}}\\
 & \leq C(q,M)\varepsilon^{q-1}.
\end{align*}
To estimate $T_{1}$, we define $h(s):=\omega_{x}^{1/2}(C_{0}^{2}s^{-1})$
like in the case $p\geq2$, where $C_{0}=C_{0}(q,M)$ is the constant
in (\ref{eq:error est 4}). Then using also both inequalities
in (\ref{eq:error est 4}) we have
\begin{align*}
\frac{\left|x-x_{\varepsilon}\right|^{3q-2}}{\varepsilon^{3(q-1)}} & =\left|\eta\right|^{2}h(\left|\eta\right|)\frac{\left|x-x_{\varepsilon}\right|^{-2(q-1)}}{\varepsilon^{-2(q-1)}}\frac{\left|x-x_{\varepsilon}\right|^{3q-2}}{\varepsilon^{3(q-1)}}\frac{1}{h(\frac{\left|x-x_{\varepsilon}\right|^{q-1}}{\varepsilon^{q-1}})}\\
 & =\left|\eta\right|^{2}h(\left|\eta\right|)\frac{\left|x-x_{\varepsilon}\right|^{q}}{\varepsilon^{q-1}}\frac{1}{h(\frac{\left|x-x_{\varepsilon}\right|^{q-1}}{\varepsilon^{q-1}})}\\
 & \leq q\left|\eta\right|^{2}h(\left|\eta\right|)\frac{\omega_{x}(\left|x-x_{\varepsilon}\right|)}{h(C_{0}\varepsilon^{\frac{(q-1)^{2}}{q}-(q-1)})}\\
 & \leq q\left|\eta\right|^{2}h(\left|\eta\right|)\frac{\omega_{x}(C_{0}\varepsilon^{\frac{q-1}{q}})}{h(C_{0}\varepsilon^{-\frac{q-1}{q}})}\\
 & =q\left|\eta\right|^{2}h(\left|\eta\right|)\frac{\omega_{x}(C_{0}\varepsilon^{\frac{q-1}{q}})}{\omega_{x}^{1/2}(C_{0}^{2}\frac{\varepsilon^{\frac{q-1}{q}}}{C_{0}})}\\
 & =q\left|\eta\right|^{2}h(\left|\eta\right|)\omega_{x}^{1/2}(C_{0}\varepsilon^{\frac{q-1}{q}}).
\end{align*}
Therefore, we obtain
\[
T_{1}=C(q)\left|\eta\right|^{p-2}\frac{\left|x-x_{\varepsilon}\right|^{3q-2}}{\varepsilon^{3(q-1)}}\leq C(q)\left|\eta\right|^{p}h(\left|\eta\right|)\omega_{x}^{1/2}(C_{0}\varepsilon^{\frac{q-1}{q}}).
\]
Combining the estimates of $T_{1}$ and $T_{2}$ with (\ref{eq:inf conv is super 3}),
we obtain
\begin{align*}
(p-1)\theta-\frac{1}{u_{\varepsilon}^{p-2}(x,t)}F(\eta,X) & \geq-E_{\varepsilon}(\left|\eta\right|^{p}h(\left|\eta\right|)+1),
\end{align*}
where we have denoted 
\begin{align*}
    E_{\varepsilon}:=C(N,p,q,m,M)\max(\omega_{x}^{1/2}(C_{0}\varepsilon^{\frac{q-1}{q}}),\varepsilon^{q-1}).
\end{align*}
This yields the desired estimate in the case $1<p<2$.

It remains to consider the case $\eta=0$. Now $x=x_{\varepsilon}$
and so taking $(z,\tau)=(x,t)$ in (\ref{eq:inf conv is super 1}), we
see that the mapping $\phi(y,s):=u(x,t_{\varepsilon})-\frac{\left|y-x\right|^{q}}{q\varepsilon^{q-1}}-\frac{\left|s-t\right|^{2}}{2\delta_{\varepsilon}}+\frac{\left|t_{\varepsilon}-t\right|^{2}}{2\delta_{\varepsilon}}$
touches $u$ from below at $(y,s)=(x,t_{\varepsilon})$. Hence, since
$u$ is a viscosity supersolution and $D\phi(y,s)\not=0$ when $y\not=x$,
we have
\[
\limsup_{(y,s)\rightarrow(x,t_{\varepsilon}), y\not=x}\left((p-1)\partial_{s}\phi(y,s)-u^{2-p}(y,s)\Delta_{p}\phi(y,s)\right)\geq0.
\]
On the other hand, since $q>p/(p-1)$ or $p>2$, we have $\Delta_{p}\phi(y,s)\rightarrow0$
as $y\rightarrow x$. Therefore, we obtain $0\leq(p-1)\partial_{s}\phi(x,t_{\varepsilon})=\frac{t-t_{\varepsilon}}{\delta_{\varepsilon}}=\theta.$
\end{proof}
Next, we use the semi-concavity of inf-convolution to show that in
the context of Lemma \ref{lem:inf conv is super}, the function $u_{\varepsilon}$
is also a weak supersolution with some error term.
\begin{lem}
\label{lem:inf conv is weak} Let $u$ be a uniformly continuous viscosity
supersolution to 
\[
(p-1)u^{p-2}\partial_{t}u-\Delta_{p}u\geq0\quad\text{in}\quad\Xi.
\]
Suppose that $0<m\leq u\leq M$ in $\Xi$ and let $u_{\varepsilon}$
be as in Lemma \ref{lem:inf conv error est}.  Then $u_{\varepsilon}$
is a weak supersolution to 
\begin{equation}
\partial_{t}u_{\varepsilon}^{p-1}-\Delta_{p}u_{\varepsilon}\geq G(Du_{\varepsilon})\quad\text{in}\quad\Xi_{\varepsilon},\label{eq:weak equation with error}
\end{equation}
where 
\begin{equation}
G(Du_{\varepsilon})=\begin{cases}
0 & \text{if }\left|Du_{\varepsilon}\right|=0,\\
-E_{\varepsilon}\left(\left|Du_{\varepsilon}\right|^{p}h(\left|Du_{\varepsilon}\right|)+\left|Du_{\varepsilon}\right|^{\max(0,p-2)}\right) & \text{if }\left|Du_{\varepsilon}\right|\not=0.
\end{cases}\label{eq:G}
\end{equation}
Here $E_{\varepsilon}$ and $h$ are as in Lemma \ref{lem:inf conv is super}.
\end{lem}

\begin{proof}
Fix non-negative $\varphi\in C_{0}^{\infty}(\Xi_{\varepsilon})$.
By Lemma \ref{lem:inf properties}, the function 
\[
\phi(x,t):=u_{\varepsilon}(x,t)-C(\left|x\right|^{2}+t^{2}),
\]
where $C=C(q, \eps)$, is concave in $\Xi_{\varepsilon}$. Therefore, we can approximate
it by smooth concave functions $\phi_{j}$ so that
\[
(\partial_{t}\phi_{j},D\phi_{j},D^{2}\phi_{j})\rightarrow(\partial_{t}\phi,D\phi,D^{2}\phi)\quad\text{a.e.\ in}\quad\Xi_{\varepsilon}.
\]
Above $D^2\phi$ denotes the Alexandrov derivative, see \cite{lindqvist25} or \cite[Section 6.4, pp.\ 242--245]{evansg92}.
We define 
\[
u_{\varepsilon,j}(x,t):=\phi_{j}(x,t)+C(\left|x\right|^{2}+t^{2}).
\]
We regularize the equation in order to remove the singularities, and to this end let $\delta\in(0,1)$. Then we compute using integration by parts
\begin{align*}
 & \int_{\Xi_{\varepsilon}}\varphi(p-1)u_{\varepsilon,j}^{p-2}\partial_{t}u_{\varepsilon,j}\d z\\
 & \hspace{1 em}+\int_{\Xi_{\varepsilon}}-\varphi(\delta+\left|Du_{\varepsilon,j}\right|^{2})^{\frac{p-2}{2}}\big(\Delta u_{\varepsilon,j}+\frac{p-2}{\delta+\left|Du_{\varepsilon,j}\right|^{2}}Du_{\varepsilon,j}^{\prime}D^{2}u_{\varepsilon,j}Du_{\varepsilon,j}\big)-\varphi G(Du_{\varepsilon,j})\d z\\
 & =\int_{\Xi_{\varepsilon}}\varphi\partial_{t}u_{\varepsilon,j}^{p-1}-\varphi\div\left((\delta+\left|Du_{\varepsilon,j}\right|^{2})^{\frac{p-2}{2}}Du_{\varepsilon,j}\right)-\varphi G(Du_{\varepsilon,j})\d z\\
 & =\int_{\Xi_{\varepsilon}}-u_{\varepsilon,j}^{p-1}\partial_{t}\varphi+(\delta+\left|Du_{\varepsilon,j}\right|^{2})^{\frac{p-2}{2}}Du_{\varepsilon,j}\cdot D\varphi-\varphi G(Du_{\varepsilon,j})\d z.
\end{align*}
We denote 
\[
\Delta_{p,\delta}u_{\varepsilon,j}:=(\delta+\left|Du_{\varepsilon,j}\right|^{2})^{\frac{p-2}{2}}\big(\Delta u_{\varepsilon,j}+\frac{p-2}{\delta+\left|Du_{\varepsilon,j}\right|^{2}}(Du_{\varepsilon,j})^{\prime}D^{2}u_{\varepsilon,j}Du_{\varepsilon,j}\big),
\]
so that by the previous computation we have
\begin{align*}
 & \liminf_{j\rightarrow\infty}\int_{\Xi_{\varepsilon}}\varphi\big((p-1)u_{\varepsilon,j}^{p-2}\partial_{t}u_{\varepsilon,j}-\Delta_{p,\delta}u_{\varepsilon,j}-G(Du_{\varepsilon,j})\big)\d z\\
 &=\lim_{j\rightarrow\infty}\int_{\Xi_{\varepsilon}}-u_{\varepsilon,j}^{p-1}\partial_{t}\varphi+(\delta+\left|Du_{\varepsilon,j}\right|^{2})^{\frac{p-2}{2}}Du_{\varepsilon,j}\cdot D\varphi-\varphi G(Du_{\varepsilon,j})\d z.
\end{align*}
We use Fatou's lemma at the left-hand side and dominated convergence
at the right-hand side. This way, we arrive at
\begin{align}\label{eq:aux}
 & \int_{\Xi_{\varepsilon}}\varphi\big((p-1)u_{\varepsilon}^{p-2}\partial_{t}u_{\varepsilon}-\Delta_{p,\delta}u_{\varepsilon}-G(Du_{\varepsilon})\big)\d z \nonumber \\
 & \leq\int_{\Xi_{\varepsilon}}-u_{\varepsilon}^{p-1}\partial_{t}\varphi+(\delta+\left|Du_{\varepsilon}\right|^{2})^{\frac{p-2}{2}}Du_{\varepsilon}\cdot D\varphi-\varphi G(Du_{\varepsilon})\d z.
\end{align}
To justify the use of Fatou's lemma, observe that by the concavity of
$\phi_{j}$ we have $D^{2}u_{\varepsilon,j}\leq CI$.
Thus the integral at the left-hand side of \eqref{eq:aux} has a lower bound when
$Du_{\varepsilon,j}=0$. When $Du_{\varepsilon,j}\not=0$, we have
\begin{align*}
-\Delta_{p,\delta}u_{\varepsilon,j} & =-\left(\delta+\left|Du_{\varepsilon,j}\right|^{2}\right)^{\frac{p-2}{2}}\left(\Delta u_{\varepsilon,j}+\frac{p-2}{\delta+\left|Du_{\varepsilon,j}\right|^{2}}Du_{\varepsilon,j}^{\prime}D^{2}u_{\varepsilon,j}D^{2}u_{\varepsilon,j}\right)\\
 & =-\frac{\left(\delta+\left|Du_{\varepsilon,j}\right|^{2}\right)^{\frac{p-2}{2}}}{\delta+\left|Du_{\varepsilon,j}\right|^{2}}\left(\left|Du_{\varepsilon,j}\right|^{2}\left(\Delta u_{\varepsilon,j}+\frac{p-2}{\left|Du_{\varepsilon,j}\right|^{2}}Du_{\varepsilon,j}^{\prime}D^{2}u_{\varepsilon,j}Du_{\varepsilon,j}\right)+\delta\Delta u_{\varepsilon,j}\right)\\
 & \geq-\frac{\left(\delta+\left|Du_{\varepsilon,j}\right|^{2}\right)^{\frac{p-2}{2}}}{\delta+\left|Du_{\varepsilon,j}\right|^{2}}C(\left|Du_{\varepsilon,j}\right|^{2}(N+p-2)+\delta N)\\
 & \geq-C\frac{\left(\delta+\left|Du_{\varepsilon,j}\right|^{2}\right)^{\frac{p-2}{2}}}{\delta+\left|Du_{\varepsilon,j}\right|^{2}}(\left|Du_{\varepsilon,j}\right|^{2}+\delta)\max(N,N+p-2)\\
 & =-C(\delta+\left|Du_{\varepsilon,j}\right|^{2})^{\frac{p-2}{2}}\max(N,N+p-2),
\end{align*}
where the last member is bounded independently of $j$ since $Du_{\varepsilon}$
is bounded. Therefore, our use of Fatou's lemma is justified.
Next, we let $\delta\rightarrow0$. By the dominated convergence theorem we have
\begin{align*}
 & \lim\inf_{\delta\rightarrow0}\int_{\Xi_{\varepsilon}}-u_{\varepsilon}^{p-1}\partial_{t}\varphi+(\delta+\left|Du_{\varepsilon}\right|^{2})^{\frac{p-2}{2}}Du_{\varepsilon}\cdot D\varphi-\varphi G(Du_{\varepsilon})\d z\\
 & =\int_{\Xi_{\varepsilon}}-u_{\varepsilon}^{p-1}\partial_{t}\varphi+\left|Du_{\varepsilon}\right|^{p-2}Du_{\varepsilon}\cdot D\varphi-\varphi G(Du_{\varepsilon})\d z,
\end{align*}
where the term $\left|Du_{\varepsilon}\right|^{p-2}Du_{\varepsilon}$
is read as zero wherever $Du_{\varepsilon}=0$. On the other hand,
by Fatou's lemma we have
\begin{align*}
 & \liminf_{\delta\rightarrow0}\int_{\Xi_{\varepsilon}}\left((p-1)u_{\varepsilon}^{p-2}\partial_{t}u_{\varepsilon}-\Delta_{p,\delta}u_{\varepsilon}-G(Du_{\varepsilon})\right)\varphi\d z\\
 & \geq\int_{\Xi_{\varepsilon}}\liminf_{\delta\rightarrow0}\left((p-1)u_{\varepsilon}^{p-2}\partial_{t}u_{\varepsilon}-\Delta_{p,\delta}u_{\varepsilon}-G(Du_{\varepsilon})\right)\varphi\d z\\
 & =\int_{\Xi_{\varepsilon}\cap\left\{ Du_{\varepsilon}\not=0\right\} }\liminf_{\delta\rightarrow0}\left((p-1)u_{\varepsilon}^{p-2}\partial_{t}u_{\varepsilon}-\Delta_{p,\delta}u_{\varepsilon}-G(Du_{\varepsilon})\right)\varphi\d z\\
 & \ \ \ +\int_{\Xi_{\varepsilon}\cap\left\{ Du_{\varepsilon}=0\right\} }\liminf_{\delta\rightarrow0}\left((p-1)u_{\varepsilon}^{p-2}\partial_{t}u_{\varepsilon}\right)\varphi\d z\\
 & \geq0.
\end{align*}
where the last estimate follows from Lemma \ref{lem:inf conv is super}. The lemma could be applied since $u_\varepsilon$ is twice differentiable almost everywhere by Alexandrov's theorem, and so for almost every $(x,t)$ in the support of $\varphi$ we have $(\partial_t u_\varepsilon (x,t), Du_\varepsilon (x,t),  D^2 u_\varepsilon (x,t))\in \mathcal P^{2,-} u(x,t)$.
We still need to justify our use of Fatou's lemma. When $Du_{\varepsilon}=0$,
this follows directly from the inequality 
\[
D^{2}u_{\varepsilon}\leq\frac{q-1}{\varepsilon}\left|Du_{\varepsilon}\right|^{\frac{q-2}{q-1}}I,
\]
which holds by Lemma \ref{lem:inf properties}. When $Du_{\varepsilon}\not=0$, we estimate
\begin{align*}
-\Delta_{p,\delta}u_{\varepsilon} & =-\left(\delta+\left|Du_{\varepsilon}\right|^{2}\right)^{\frac{p-2}{2}}\left(\Delta u_{\varepsilon}+\frac{p-2}{\delta+\left|Du_{\varepsilon}\right|^{2}}Du_{\varepsilon}^{\prime}D^{2}u_{\varepsilon}Du_{\varepsilon}\right)\\
 & =-\frac{(\delta+\left|Du_{\varepsilon}\right|^{2})^{\frac{p-2}{2}}}{\delta+\left|Du_{\varepsilon}\right|^{2}}\left(\left|Du_{\varepsilon}\right|^{2}\left(\Delta u_{\varepsilon}+\frac{p-2}{\left|Du_{\varepsilon}\right|}Du_{\varepsilon}^{\prime}D^{2}u_{\varepsilon}Du_{\varepsilon}\right)+\delta\Delta u_{\varepsilon}\right)\\
 & \geq-\frac{(\delta+\left|Du_{\varepsilon}\right|^{2})^{\frac{p-2}{2}}}{\delta+\left|Du_{\varepsilon}\right|^{2}}\frac{q-1}{\varepsilon}\left(\left|Du_{\varepsilon}\right|^{\frac{q-2}{q-1}+2}(N+p-2)+\left|Du_{\varepsilon}\right|^{\frac{q-2}{q-1}}\delta N\right)\\
 & \geq-\frac{(\delta+\left|Du_{\varepsilon}\right|^{2})^{\frac{p-2}{2}}}{\delta+\left|Du_{\varepsilon}\right|^{2}}\frac{q-1}{\varepsilon}\left((\left|Du_{\varepsilon}\right|^{2}+\delta)\left|Du_{\varepsilon}\right|^{\frac{q-2}{q-1}+2}\max(N,N+p-2)\right)\\
 & \geq-(\delta+\left|Du_{\varepsilon}\right|^{2})^{\frac{p-2}{2}}\left|Du_{\varepsilon}\right|^{\frac{q-2}{q-1}}\frac{(q-1)}{\varepsilon}\max(N,N+p-2)\\
 & \geq-(\delta+\left|Du_{\varepsilon}\right|^{2})^{\frac{1}{2}(p-2+\frac{q-2}{q-1})}\frac{(q-1)}{\varepsilon}\max(N,N+p-2),
\end{align*}
which is bounded independently of $\delta$ since $p-2+\frac{q-2}{q-1}>0$.
\end{proof}

\subsection{Local energy estimates and convergence of $Du_{\varepsilon}$ in
$L_{loc}^{r}$.}

\label{sec:energy-estimates}

Next we establish an energy estimate for weak supersolutions to the
equation (\ref{eq:weak equation with error}). Observe that the growth
of the first-order term $\left|\eta\right|^{p}h(\eta)$ may be stronger
than $\left|\eta\right|^{q}$ for any $q<p$. As in the proof of an
energy estimate one usually tests with the solution itself, the term
$u\left|Du\right|^{p}h(Du)$ has to be dealt with. This could be difficult
if we only knew that $u$ is in some Sobolev space. However, since
we are dealing with locally bounded supersolutions, this is not a
problem for us.

If $u_{\varepsilon}$ is the inf-convolution of a viscosity supersolution
to (\ref{eq:weak equation with error}), then the energy estimate
together with Lemma \ref{lem:inf conv is weak} guarantees that $Du_{\varepsilon}$
converges weakly in $L_{loc}^{p}(\Xi)$ as $\varepsilon\rightarrow0$
up to a subsequence. However, eventually we need to pass to the limit
under the integral sign of
\[
\int_{\Xi}-u_{\varepsilon}^{p-1}\partial_{t}\varphi+\left|Du_{\varepsilon}\right|^{p-2}Du_{\varepsilon}\cdot D\varphi+\varphi G(Du_{\varepsilon})\d z,
\]
where $\varphi$ is an arbitrary fixed test function and $G$ is as
in Lemma \ref{lem:inf conv is weak}. For this we need stronger convergence,
which we shall establish in Lemma \ref{lem:strong conv}. Since we apply this lemma to 
$\inf$-convolutions that are locally Lipschitz, the Lipschitz assumption in the next lemma does not impose any additional restriction for our purposes.
\begin{lem}[Caccioppoli's estimate]
\label{lem:caccioppoli} Assume that $u$ is a locally Lipschitz
continuous weak supersolution to 
\[
\partial_{t}u^{p-1}-\Delta_{p}u\geq G(Du)\quad\text{in \ensuremath{\Xi}, }
\]
where $G$ is defined by the formula (\ref{eq:G}). 
Suppose moreover
that $0\leq u\leq M$ in $\Xi$. Then
\[
\int_{\Xi}\left|Du\right|^{p}\xi^{p}\d z\leq C(p,M,h)\int_{\Xi}\left|\partial_{t}\xi^{p}\right|+\left|D\xi\right|^{p}+\xi^{p}\d z
\]
for any cut-off function $\xi\in C_{0}^{\infty}(\Xi),\ \xi\ge 0.$
\end{lem}

\begin{proof}
Define $\varphi:=\xi^{p}(M-u)\ge 0.$ By local Lipschitz continuity (both in space and time) $\varphi$
is an admissible test function whose time derivative exist a.e., and we obtain using the definition of
$G$
\begin{align}
\int_{\Xi}\xi^{p}\left|Du\right|^{p}\d z & \leq\int_{\Xi}u^{p-1}\partial_{t}\varphi+p\xi^{p-1}(M-u)\left|Du\right|^{p-1}\left|D\xi\right|\d z\nonumber \\
 & \ \ \ +\int_{\Xi\cap\left\{ Du\not=0\right\} }\varphi E_{\varepsilon}(\left|Du\right|^{p}h(\left|Du\right|)+\left|Du\right|^{\max(p-2,0)})\d z.\label{eq:energy est 1}
\end{align}
We proceed to estimate the terms on the right-hand side. We have by
integration by parts
\begin{align}
\int_{\Xi}u^{p-1}\partial_{t}((M-u)\xi^{p})\d z & =\int_{\Xi}-\frac{1}{p}\partial_{t}u^{p}\xi^{p}+u^{p-1}(M-u)\partial_{t}\xi^{p}\d z\nonumber\\
 & =\int_{\Xi}\frac{1}{p}u^{p}\partial_{t}\xi^{p}+u^{p-1}(M-u)\partial_{t}\xi^{p}\d z\nonumber \\
 & \leq\int_{\Xi}C(p)M^{p}\left|\partial_{t}\xi^{p}\right|\d z.\label{eq:energy est 2}
\end{align}
Using Young's inequality we obtain
\begin{align}
\int_{\Xi}p\xi^{p-1}(M-u)\left|Du\right|^{p-1}\left|D\xi\right|\d z\leq & \int_{\Xi}\frac{1}{4}\xi^{p}\left|Du\right|^{p}\d z+C(p)\int_{\Xi}M^{p}\left|D\xi\right|^{p}\d z.\label{eq:energy est 3}
\end{align}
Further, if $p>2$, using Young's inequality with the exponents $p/(p-2)$
and $p/2$, we get 
\begin{align}
\int_{\Xi\cap\left\{ Du_{\varepsilon}\not=0\right\} }\varphi E_{\varepsilon}\left|Du\right|^{\max(p-2,0)}\d z & =\int_{\Xi}E_{\varepsilon}\xi^{p-2}\left|Du\right|^{p-2}\xi^{2}(M-u)\d z\nonumber \\
 & \leq\int_{\Xi}\frac{1}{4}\xi^{p}\left|Du\right|^{p}\d z+\int_{\Xi}\xi^{p}E_{\varepsilon}^{p/2}M^{p/2}\d z.\label{eq:energy est 4}
\end{align}
Finally, since $h(s)\rightarrow0$ as $s\rightarrow \infty$, there exists
$S>0$ such that $h(s)\leq\frac{1}{4M}$ when $s>S$. Using this,
we estimate
\begin{align}
\int_{\Xi}E_{\varepsilon}\varphi\left|Du\right|^{p}h(\left|Du\right|)\d z & \leq\int_{\Xi\cap\left\{ \left|Du\right|\leq S\right\} }E_{\varepsilon}M\left|Du\right|^{p}h(0)\xi^{p}\d z\nonumber \\
 & \ \ \ +\int_{\Xi\cap\left\{ \left|Du\right|>S\right\} }E_{\varepsilon}M\left|Du\right|^{p}\frac{1}{4M}\xi^{p}\d z\nonumber \\
 & \leq\int_{\Xi\cap\left\{ \left|Du\right|\leq S\right\} }E_{\varepsilon}MS^{p}h(0)\xi^{p}\d z+\int_{\Xi\cap\left\{ \left|Du\right|>S\right\} }\frac{1}{4}\left|Du\right|^{p}\xi^{p}\d z.\label{eq:energy est 5}
\end{align}
Combining the estimates (\ref{eq:energy est 2})-(\ref{eq:energy est 5})
with (\ref{eq:energy est 1}) and absorbing the term with $\frac{1}{4}\left|Du\right|^{p}$
to the left-hand side, we arrive at
\begin{align*}
\int_{\Xi}\frac{1}{2}\left|Du\right|^{p}\xi^{p}\d z & \leq\int_{\Xi}C(p)M^{p}\left|\partial_{t}\xi^{p}\right|+C(p)M^{p}\left|D\xi\right|^{p}+E_{\varepsilon}MS^{p}h(0)\xi^{p}\d z\\
 & \ \ \ +\begin{cases}
\int_{\Xi}\xi^{p}ME_{\varepsilon}\d z, & 1<p\leq2,\\
\int_{\Xi}\frac{1}{4}\xi^{p}\left|Du\right|^{p}+\xi^{p}E_{\varepsilon}^{p/2}M^{p/2}\d z, & p>2.
\end{cases}
\end{align*}
In the case $p>2$ we absorb the additional term with $\frac{1}{4}\left|Du\right|^{p}$
to the left-hand side. The claimed estimate follows.
\end{proof}

\begin{lem}
\label{lem:strong conv} Let $(u_{l})$ be a sequence of locally Lipschitz
continuous weak supersolutions to
\[
\partial_{t}u^{p-1}-\Delta_{p}u\geq G(Du)\quad\text{in}\quad\Xi
\]
where $G$ is defined by the formula (\ref{eq:G}). 
Suppose that
$u_{l}\rightarrow u$ pointwise almost everywhere in $\Xi$. Assume moreover that for
all $l\in\mathbb{N}$ we have
\[
0<m\leq u_{l} \leq M<\infty\quad\text{in}\quad\Xi.
\]
 Then $(Du_{l})$ is a Cauchy sequence in $L_{loc}^{r}(\Xi)$ for
any $1<r<p$.
\end{lem}

\begin{proof}
Let $U\Subset\Xi$ and take a cut-off function $\zeta\in C_{0}^{\infty}(\Xi)$
such that $0\leq\zeta\leq1$ and $\zeta\equiv1$ in $U.$ For $\delta>0,$
set
\[
w_{jk}:=\begin{cases}
\delta, & u_{j}^{p-1}-u_{k}^{p-1}>\delta.\\
u_{j}^{p-1}-u_{k}^{p-1}, & -\delta<u_{j}^{p-1}-u_{k}^{p-1}<\delta,\\
-\delta, & u_{j}^{p-1}-u_{k}^{p-1}<-\delta.
\end{cases}
\]
Then $(\delta-w_{jk})\zeta\ge 0$ is an admissible test function by local
Lipschitz continuity and the fact that $w_{jk}\in[-\delta,\delta]$. Moreover, the test function has time derivatives almost everywhere. Using it as a test function for $u_{j}$, we
obtain
\begin{align}
0 & \leq\int_{\Xi}-u_{j}^{p-1}\partial_{t}\big( (\delta-w_{jk})\zeta\big)+\left|Du_{j}\right|^{p-2}Du_{j}\cdot D((\delta-w_{jk})\zeta)\d z\nonumber \\
 & \ \ \ +\int_{\Xi\cap\left\{ Du_{j}\not=0\right\} }-E_{\varepsilon}(\delta-w_{jk})\zeta\big(\left|Du_{j}\right|^{p}h(\left|Du_{j}\right|)+\left|Du_{j}\right|^{\max(0,p-2)}\big)\d z\nonumber \\
 & =\int_{\Xi}-\zeta\left|Du_{j}\right|^{p-2}Du_{j}\cdot Dw_{jk}+(\delta-w_{jk})\left|Du_{j}\right|^{p-2}Du_{j}\cdot D\zeta\d z\nonumber \\
 & \ \ \ +\int_{\Xi\cap\left\{ Du_{j}\not=0\right\} }-E_{\varepsilon}(\delta-w_{jk})\zeta\big(\left|Du_{j}\right|^{p}h(\left|Du_{j}\right|)+\left|Du_j\right|^{\max(0,p-2)}\big)\d z\nonumber \\
 & \ \ \ +\int_{\Xi}\zeta u_{j}^{p-1}\partial_{t}w_{jk}-u_{j}^{p-1}(\delta-w_{jk})\partial_{t}\zeta\d z.\label{eq:strong conv 0}
\end{align}
Then, since
\[
Dw_{jk}=(p-1)\chi_{\left\{ \left|u_{j}^{p-1}-u_{k}^{p-1}\right|<\delta\right\} }(u_{j}^{p-2}Du_{j}-u_{k}^{p-2}Du_{k}),
\]
and $\left|w_{jk}\right|\leq\delta$, it follows from (\ref{eq:strong conv 0})
that
\begin{align*}
 & \int_{\Xi\cap\left\{ \left|u_{j}^{p-1}-u_{k}^{p-1}\right|<\delta\right\} }\zeta(p-1)\left|Du_{j}\right|^{p-2}Du_{j}\cdot(u_{j}^{p-2}Du_{j}-u_{k}^{p-2}Du_{k})\d z\\
 & \leq\int_{\Xi\cap\left\{ Du_{j}\not=0\right\} }2\delta\left|Du_{j}\right|^{p-1}\left|D\zeta\right|+2\delta E_{\varepsilon}\zeta(\left|Du_{j}\right|^{p}h(\left|Du_{j}\right|)+\left|Du_{j}\right|^{\max(0,p-2)})\d z\\
 & \ \ \ +\int_{\Xi}\zeta u_{j}^{p-1}\partial_{t}w_{jk}+2\delta u_{j}^{p-1}\left|\partial_{t}\zeta\right|\d z.
\end{align*}
Since $u_{k}$ is a weak supersolution, the same arguments but
testing now with $(\delta+w_{jk})\zeta\ge 0$ yields
\begin{align*}
 & \int_{\Xi\cap\left\{ \left|u_{j}^{p-1}-u_{k}^{p-1}\right|<\delta\right\} }-\zeta(p-1)\left|Du_{k}\right|^{p-2}Du_{k}\cdot(u_{j}^{p-2}Du_{j}-u_{k}^{p-2}Du_{k})\d z\\
 & \le\int_{\Xi\cap\left\{ Du_{k}\not=0\right\} }2\delta\left|Du_{k}\right|^{p-1}\left|D\zeta\right|+2\delta E_{\varepsilon}\zeta(\left|Du_{k}\right|^{p}h(\left|Du_{k}\right|)+\left|Du_{k}\right|^{\max(0,p-2)})\d z\\
 & \ \ \ +\int_{\Xi}-\zeta u_{k}^{p-1}\partial_{t}w_{jk}+2\delta u_{k}^{p-1}\left|\partial_{t}\zeta\right|\d z.
\end{align*}
The idea of choosing the above test functions is to get, combining the last two inequalities, the following estimate where the first term contains a useful difference
\begin{align*}
 & (p-1)\int_{\Xi\cap\left\{ \left|u_{j}^{p-1}-u_{k}^{p-1}\right|<\delta\right\} }\zeta\big(\left|Du_{k}\right|^{p-2}Du_{k}-\left|Du_{j}\right|^{p-2}Du_{j}\big)\cdot(u_{j}^{p-2}Du_{j}-u_{k}^{p-2}Du_{k})\d z\\
 & \leq\int_{\Xi}\zeta(u_{j}^{p-1}-u_{k}^{p-1})\partial_{t}w_{jk}\d z\\
 & \ \ \ +\int_{\Xi}2\delta(u_{j}^{p-1}+u_{k}^{p-1})\left|\partial_{t}\zeta\right|\d z\\
 & \ \ \ +\sum_{l=j,k}\int_{\Xi\cap\left\{ Du_{l}\not=0\right\} }2\delta\left|D\zeta\right|\left|Du_{l}\right|^{p-1}+2\delta E_{\varepsilon}\zeta\big(\left|Du_{l}\right|^{p}h(\left|Du_{l}\right|)+\left|Du_{l}\right|^{\max(0,p-2)}\big)\d z.
\end{align*}
 At the left-hand side of the above estimate we have
\begin{align*}
 & \int_{\Xi\cap\left\{ \left|u_{j}^{p-1}-u_{k}^{p-1}\right|<\delta\right\} }\zeta\big(\left|Du_{j}\right|^{p-2}Du_{j}-\left|Du_{k}\right|^{p-2}Du_{k}\big)\cdot(u_{j}^{p-2}Du_{j}-u_{k}^{p-2}Du_{k})\d z\\
 & =\int_{\Xi\cap\left\{ \left|u_{j}^{p-1}-u_{k}^{p-1}\right|<\delta\right\} }\zeta u_{j}^{p-2}\big(\left|Du_{j}\right|^{p-2}Du_{j}-\left|Du_{k}\right|^{p-2}Du_{k}\big)\cdot(Du_{j}-Du_{k})\d z\\
 & \ \ \ -\int_{\Xi\cap\left\{ \left|u_{j}^{p-1}-u_{k}^{p-1}\right|<\delta\right\} }\zeta(u_{k}^{p-2}-u_{j}^{p-2})(\left|Du_{j}\right|^{p-2}Du_{j}-\left|Du_{k}\right|^{p-2}Du_{k})\cdot Du_{k}\d z.
\end{align*}
Combining the last two displays we have
\begin{align}
 & (p-1)\int_{\Xi\cap\left\{ \left|u_{j}^{p-1}-u_{k}^{p-1}\right|<\delta\right\} }\zeta u_{j}^{p-2}(\left|Du_{j}\right|^{p-2}Du_{j}-\left|Du_{k}\right|^{p-2}Du_{k})\cdot(Du_{j}-Du_{k})\d z\nonumber \\
 & \leq\int_{\Xi}\zeta(u_{j}^{p-1}-u_{k}^{p-1})\partial_{t}w_{jk}\d z\nonumber \\
 & \ \ \ +\int_{\Xi}2\delta(u_{j}^{p-1}+u_{k}^{p-1})\left|\partial_{t}\zeta\right|\d z\nonumber \\
 & \ \ \ +\sum_{l=j,k}\int_{\Xi\cap\left\{ Du_{l}\not=0\right\} }2\delta\left|D\zeta\right|\left|Du_{l}\right|^{p-1}+2\delta E_{\varepsilon}\zeta(\left|Du_{l}\right|^{p}h(\left|Du_{l}\right|)+\left|Du_{l}\right|^{\max(0,p-2)})\d z\nonumber \\
 & \ \ \ +(p-1)\int_{\Xi\cap\left\{ \left|u_{j}^{p-1}-u_{k}^{p-1}\right|<\delta\right\} }\zeta(u_{k}^{p-2}-u_{j}^{p-2})(\left|Du_{j}\right|^{p-2}Du_{j}-\left|Du_{k}\right|^{p-2}Du_{k})\cdot Du_{k}\d z\nonumber \\
 & =:I_{1}+I_{2}+I_{3}+I_{4},\label{eq:strong conv 2}
\end{align}
where the definition of $I_{1},\ldots,I_{4}$ is again apparent. We
now proceed to estimate the right-hand side. 

\textbf{Estimate of $I_{1}$:} Using the definition of $w_{jk}$ and
integrating by parts we get 
\begin{align*}
I_{1} & =\int_{\Xi\cap\left\{ \left|u_{j}^{p-1}-u_{k}^{p-1}\right|<\delta\right\} }\zeta(u_{j}^{p-1}-u_{k}^{p-1})\partial_{t}(u_{j}^{p-1}-u_{k}^{p-1})\d z\\
 & =\int_{\Xi\cap\left\{ \left|u_{j}^{p-1}-u_{k}^{p-1}\right|<\delta\right\} }\frac{1}{2}\zeta\partial_{t}(u_{j}^{p-1}-u_{k}^{p-1})^{2}\d z\leq\int_{\Xi}\frac{1}{2}w_{jk}^{2}\left|\partial_{t}\zeta\right|\d z\leq\delta^{2}C(\zeta)
\end{align*}

\textbf{Estimate of $I_{2}$:} Since $\left|u_{l}\right|\le M$ in
$\Xi$ for all $l\in\mathbb{N}$, we have 
\[
I_{2}=\int_{\Xi}2\delta(\left|u_{j}\right|^{p-1}+\left|u_{k}\right|^{p-1})\left|\partial_{t}\zeta\right|\d z\leq\delta C(p,M,\zeta).
\]

\textbf{Estimate of $I_{3}$:} Observe that by Lemma \ref{lem:caccioppoli}
we have
\begin{align}
\sup_{l\in\mathbb{N}}\int_{\supp\zeta}\left|Du_{l}\right|^{p}\d z & \leq C(p,h,M,\zeta).\label{eq:strong conv gradient bound}
\end{align}
Since $h$ is a bounded function and $E_{\varepsilon}\in(0,1]$, the
estimate (\ref{eq:strong conv gradient bound}) together with H{\"o}lder's
inequality readily implies
\begin{align*}
I_{3} & =\sum_{l=j,k}\int_{\Xi\cap\left\{ Du_{l}\not=0\right\} }2\delta\left|D\zeta\right|\left|Du_{l}\right|^{p-1}+2\delta E_{\varepsilon}\zeta(\left|Du_{l}\right|^{p}h(\left|Du_{l}\right|)+\left|Du_{l}\right|^{\max(0,p-2)})\d z\\
 & \leq\delta C(p,h,M,\zeta).
\end{align*}

\textbf{Estimate of $I_{4}$:} Observe that since $0<m\leq u_{l}\leq M$
in $\Xi$ for all $l\in\mathbb{N}$, we have
\[
\left|u_{k}^{p-2}-u_{j}^{p-2}\right|\leq C(p,m,M)\left|u_{j}^{p-1}-u_{k}^{p-1}\right|\quad\text{in\ensuremath{\quad}}\Xi,
\]
using the mean value theorem.
Using this and (\ref{eq:strong conv gradient bound}), we estimate
\begin{align*}
I_{4}= & (p-1)\int_{\Xi\cap\left\{ \left|u_{j}^{p-1}-u_{k}^{p-1}\right|<\delta\right\} }\zeta(u_{k}^{p-2}-u_{j}^{p-2})(\left|Du_{j}\right|^{p-2}Du_{j}-\left|Du_{k}\right|^{p-2}Du_{k})\cdot Du_{k}\d z\\
 & \leq C(p,m,M)\delta\left(\left(\int_{\supp\zeta}\left|Du_{j}\right|^{p}\d z\right)^{1/p}\left(\int_{\supp\zeta}\left|Du_{k}\right|^{p}\d z\right)^{1/p}+\int_{\supp\zeta}\left|Du_{k}\right|^{p}\d z\right)\\
 & \leq C(p,h,m,M,\zeta)\delta.
\end{align*}
Combining the estimates of $I_{1},\ldots,I_{4}$ with (\ref{eq:strong conv 2}),
we obtain
\begin{equation}
\int_{\Xi\cap\left\{ \left|u_{j}^{p-1}-u_{k}^{p-1}\right|<\delta\right\} }\zeta u_{j}^{p-2}(\left|Du_{j}\right|^{p-2}Du_{j}-\left|Du_{k}\right|^{p-2}Du_{k})\cdot(Du_{j}-Du_{k})\d z\leq\delta C_{0},\label{eq:bla1}
\end{equation}
where $C_{0}=C_{0}(p,h,M,m,\zeta)$. We now estimate the left-hand
side. If $1<p<2$, we use the vector inequality (see \cite[p98]{lindqvist_plaplace})
\[
\left(\left|a\right|^{p-2}a-\left|b\right|^{p-2}b\right)\cdot(a-b)\geq(p-1)\left|a-b\right|^{2}\left(1+\left|a\right|^{2}+\left|b\right|^{2}\right)^{\frac{p-2}{2}}
\]
estimate (\ref{eq:bla1}) and H\"older's inequality to get that (recall $\zeta\equiv1$ in $U$)
\begin{align}
 & \int_{U\cap\left\{ \left|u_{j}^{p-1}-u_{k}^{p-1}\right|<\delta\right\} }\left|Du_{j}-Du_{k}\right|^{r}\d z\nonumber \\
 & \leq\Bigg(\int_{U\cap\left\{ \left|u_{j}^{p-1}-u_{k}^{p-1}\right|<\delta\right\} }\left(1+\left|Du_{j}\right|^{2}+\left|Du_{k}\right|^{2}\right)^{\frac{r(2-p)}{2(2-r)}}\d z\Bigg)^{\frac{2-r}{2}}\nonumber \\
 & \ \ \ \cdot\Bigg(\int_{\Xi\cap\left\{ \left|u_{j}^{p-1}-u_{k}^{p-1}\right|<\delta\right\} }\frac{\zeta\left|Du_{j}-Du_{k}\right|^{2}}{\left(1+\left|Du_{j}\right|^{2}+\left|Du_{k}\right|^{2}\right)^{\frac{2-p}{2}}}\d z\Bigg)^{\frac{r}{2}}\nonumber \\
 & \leq C(p,h,r,m,M,\zeta)C_{0}\delta^{\frac{r}{2}},\label{eq:bla2}
\end{align}
where it was also used that $r(2-p)/(2-r)<p$. If $p\geq2$, using
the inequality (see \cite[p95]{lindqvist_plaplace})
\[
(\left|a\right|^{p-2}a-\left|b\right|^{p-2}b)\cdot(a-b)\geq2^{2-p}\left|a-b\right|^{p}
\]
and (\ref{eq:bla1}) we get
\begin{align}
 & \int_{U\cap\left\{ \left|u_{j}^{p-1}-u_{k}^{p-1}\right|<\delta\right\} }\left|Du_{j}-Du_{k}\right|^{r}\d z\nonumber \\
 & \leq C\left|\supp\zeta\right|^{\frac{1-r}{p}}\left(\int_{U\cap\left\{ \left|u_{j}^{p-1}-u_{k}^{p-1}\right|<\delta\right\} }\zeta\left|Du_{j}-Du_{k}\right|^{p}\d z\right)^{\frac{r}{p}}\nonumber \\
 & \leq C(p,r,m)\left(\int_{U\cap\left\{ \left|u_{j}^{p-1}-u_{k}^{p-1}\right|<\delta\right\} }\zeta u_{j}^{p-2}\left|Du_{j}-Du_{k}\right|^{p}\d z\right)^{\frac{r}{p}}\nonumber \\
 & \leq C(p,r,m)C_{0}\delta^{\frac{r}{p}}.\label{eq:bla3}
\end{align}
Combining the estimates (\ref{eq:bla2}) and (\ref{eq:bla3}), we
arrive at
\[
\int_{U\cap\left\{ \left|u_{j}^{p-1}-u_{k}^{p-1}\right|<\delta\right\} }\left|Du_{j}-Du_{k}\right|^{r}\d z\leq C(p,h,r,m,M,\zeta)\delta^{\frac{r}{\max(p,2)}}.
\]
On the other hand, by H{\"o}lder's and Chebyshev's inequalities and estimate
(\ref{eq:strong conv gradient bound}),
\begin{align*}
 & \int_{U\cap\left\{ \left|u_{j}^{p-1}-u_{k}^{p-1}\right|\geq\delta\right\} }\left|Du_{j}-Du_{k}\right|^{r}\d z\\
 & \leq\left|U\cap\left\{ \left|u_{j}^{p-1}-u_{k}^{p-1}\right|\geq\delta\right\} \right|^{\frac{p-r}{p}}\left(\int_{U\cap\left\{ \left|u_{j}^{p-1}-u_{k}^{p-1}\right|\geq\delta\right\} }\left|Du_{j}-Du_{k}\right|^{p}\d z\right)^{\frac{r}{p}}\\
 & \leq\left(\frac{1}{\delta}\int_{U}\left|u_{j}^{p-1}-u_{k}^{p-1}\right|\d z\right)^{\frac{p-r}{p}}C(p,r,h,M,\zeta).
\end{align*}
By the last two displays, we have now established that
\[
\int_{U}\left|Du_{j}-Du_{k}\right|^{r}\d z\leq\left(\left(\frac{1}{\delta}\int_{U}\left|u_{j}^{p-1}-u_{k}^{p-1}\right|\d z\right)^{\frac{p-r}{p}}+\delta^{\frac{r}{\max(p,2)}}\right)C(p,h,r,m,M,\zeta).
\]
Taking first small $\delta$ and then large $j$ and $k$, we see
that the right-hand side can be made arbitrarily small. Thus $(Du_{l})$
is a Cauchy sequence in $L_{loc}^{r}(\Xi)$ .
\end{proof}

\subsection{Passing to the limit}
\label{sec:limit}

It remains to pass to the limit $\varepsilon\rightarrow0$ to show                       
that viscosity supersolutions are weak supersolutions.
\begin{thm}
\label{thm:visc is weak} Let $u$ be a continuous viscosity
supersolution to (\ref{eq:trudinger}) in $\Xi$. Suppose that $0<m\leq u\leq M$.
Then $u$ is a local weak supersolution to (\ref{eq:trudinger}) in
$\Xi$.
\end{thm}

\begin{proof}
First, we may suppose that $u$ is uniformly continuous in $\Xi$ by taking a smaller domain if necessary. Then, fix domains $U\Subset U^{\prime}\Subset\Xi$. Let $u_{\varepsilon}$
be the inf-convolution of $u_{\varepsilon}$ and let $\varepsilon>0$
be so small that $U^{\prime}\Subset\Xi_{\varepsilon}$. By Lemma \ref{lem:inf conv is weak}
the inf-convolution is a weak supersolution to 
\begin{equation}
\partial_{t}u_{\varepsilon}^{p-1}-\Delta_{p}u_{\varepsilon}+G(Du_{\varepsilon})\geq0\quad\text{in }U^{\prime}.\label{eq:main thm 1}
\end{equation}
 Then by Lemma \ref{lem:caccioppoli} we have $Du_{\varepsilon}$
bounded in $L^{p}(U)$. Thus, passing to a subsequence if necessary,
we have $Du_{\varepsilon}\rightarrow g$ weakly in $L^{p}(U)$ for
some $g\in L^{p}(U)$. Then for any $i\in\left\{ 1,\ldots,N\right\} $
and $\varphi\in C_{0}^{\infty}(U)$ we have
\[
\int_{U}u\partial_{i}\varphi\d z=\lim_{\varepsilon\rightarrow0}\int_{U}u_{\varepsilon}\partial_{i}\varphi\d z=-\lim_{\varepsilon\rightarrow0}\int_{U}\varphi\partial_{i}u_{\varepsilon}\d z=-\int_{U}\varphi\partial_{i}u\d z
\]
so that $u$ has a weak distributional space derivative $Du=g\in L^{p}(U)$.

Let $\varphi\in C_{0}^{\infty}(U)$. It now remains to pass to the
limit in 
\[
\int_{U}-u_{\varepsilon}^{p-1}\partial_{t}\varphi+\left|Du_{\varepsilon}\right|^{p-2}Du_{\varepsilon}\cdot D\varphi\d z+\int_{U\cap\left\{ Du_{\varepsilon}\not=0\right\} }E_{\varepsilon}(\left|Du_{\varepsilon}\right|^{p}h(\left|Du_{\varepsilon}\right|)+\left|Du_{\varepsilon}\right|^{\max(0,p-2)})\varphi\d z.
\]
Convergence of the first order terms is clear since $u_{\varepsilon}\rightarrow u$
pointwise, $E_{\varepsilon}\rightarrow0$ and $\left|Du_{\varepsilon}\right|$
is bounded in $L^{p}(U)$. For the second-order term, one uses that
$Du_{\varepsilon}\rightarrow Du$ in $L^{r}(U)$ for any $1<r<p$ from Lemma~\ref{lem:strong conv}
together with the vector inequality 
\[
\left|\left|a\right|^{p-2}a-\left|b\right|^{p-2}b\right|\leq\begin{cases}
2^{2-p}\left|a-b\right|^{p-1}, & 1<p<2,\\
(p-1)(\left|a\right|^{p-2}+\left|b\right|^{p-2})\left|a-b\right|, & p\geq2,
\end{cases}
\]
which can be found on \cite[pp97-98]{lindqvist_plaplace}.
This way, by (\ref{eq:main thm 1}), we arrive at
\[
\int_{U}\left(-u^{p-1}\partial_{t}\varphi+\left|Du\right|^{p-2}Du\cdot D\varphi\right)\d z\geq0.\qedhere
\]
\end{proof}

\begin{thm}\label{thm:equivalence}
\label{thm:visc equiv weak} Suppose that $0<m\leq u\leq M$.
Then a viscosity solution $u$ to (\ref{eq:trudinger}) is a weak solution and vice versa.
\end{thm}
\begin{proof}
The theorem follows by the previous result and by recalling Theorem 3 in \cite{lindgrenl22} where they show that weak solutions are viscosity solutions. In \cite{lindgrenl22} they assume $p\ge 2$ but under our positivity assumptions this is not needed.

For the convenience of the reader, we provide the details in the case $1<p<2$.
Suppose on the contrary that there is $\varphi\in C^{2}$ that touches $u$ from below at $(x,t)$, $D\varphi(y,s)\not=0$ for
$y\not=x$, and
\[
\limsup_{(y,s)\rightarrow (x,t), y\not=x}\left(\partial_{t} \varphi(y,s)-\frac{1}{u^{p-2}(x,t)}\Delta_{p}\varphi(y,s)\right)<-\varepsilon.
\]
It follows  that
\[
(p-1)\partial_{t}\varphi-\frac{1}{u^{p-2}(x,t)}\Delta_{p}\varphi<-\delta
\]
in a neighborhood of $(x,t)$. Since $\varphi$ is $C^{2}$, $\varphi(x,t)=u(x,t)$
and $0<m\leq u\leq M$, taking $r>0$ to be small enough we may ensure
that $\varphi > 0$ and
\[
(p-1)\varphi^{2-p}\partial_{t}\varphi-\Delta_{p}\varphi<0
\]
in a cylindrical neighborhood $(B_{r}(x)\setminus\left\{ x\right\} )\times(t-r,t+r)$.
Multiplying this by non-negative test function $\phi\in C_{0}^{\infty}((B_{r}(x)\setminus\left\{ x\right\} )\times(t-r,t+r))$
and integrating, we obtain for any $0<\epsilon<r$ that
\begin{align*}
0 & >\int_{t-r}^{t+r}\int_{B_{r}(x)\setminus B_{\epsilon}(x)} \left(\phi\partial_{t}\varphi^{p-1}-\phi\Delta_{p}\varphi \right)\d y\d s \\& =\int_{t-r}^{t+r}\int_{B_{r}(x)\setminus B_{\epsilon}(x)}\left(-\varphi^{p-1}\partial_{t}\phi+\left|D\varphi\right|^{p-2}D\varphi\cdot D\phi\right)\d y\d s\\
 & \ \ \ +\int_{t-r}^{t+r}\int_{\partial B_{\epsilon}(x)}\left(\varphi\left|D\phi\right|^{p-2}D\phi\cdot\frac{y-x}{\epsilon}\right)\d S(y)\d s.
\end{align*}
Letting $\epsilon\rightarrow0$, this yields
\[
\int_{B_{r}(x_{0})\times(t-r,t+r)}\left(-\varphi^{p-1}\partial_{t}\phi+\left|D\varphi\right|^{p-2}D\varphi\cdot D\phi\right)\d z<0
\]
so that $\varphi$ is a weak subsolution in $B_{r}(x)\times(t-r,t+r)$.
But then so is 
\[
\tilde{\varphi}(y,s):=h\varphi(y,s),\quad\text{where }h:=\inf_{(y,s)\in\partial_{\mathcal{P}}B_{r}(x_{0})\times(t-r,t+r)}\frac{u(y,s)}{\varphi(y,s)}>1,
\]
and $\partial_\mathcal{P}$ denotes the parabolic boundary. Now by comparison principle of weak solutions \cite[Corollary 1]{lindgrenl22}, we have $\tilde{\varphi}\leq u$
in $B_{r}(x)\times(t-r,t+r)$, but this contradicts the fact that
$\tilde{\varphi}(x,t)=h\varphi(x,t)=hu(x,t)>u(x,t)$.
\end{proof}

As an application, and we state this just for curiosity, we also get that the limit  of a positive, equicontinuous and bounded  sequence of supersolutions (either viscosity or weak) is a  supersolution. For the $p$-parabolic equation this is a special case of Theorem 5.3 in \cite{korteKuusiParv10}, see also \cite{lindqvistm07}. 
\begin{corollary}
    Suppose that $u_i,\ i=1,2,\ldots$
is a sequence of equicontinuous (weak or viscosity) supersolutions such that $0< m\le  u_i \le M$ and $u_i$ converges to $u$
almost everywhere and thus locally uniformly. Then $u$ is a (weak and viscosity) supersolution.
\end{corollary}
\begin{proof}
Let $\Xi^\prime \Subset \Xi$ and $u_i$ be a sequence of weak supersolutions satisfying the assumptions. By \cite{lindgrenl22}, each $u_{i}$
is a viscosity supersolution to \eqref{eq:trudinger} in $\Xi$. Therefore, for small enough $\varepsilon$ it follows from Lemma \ref{lem:inf conv is weak} that the inf-convolution
$u_{i,\varepsilon}$ is a weak supersolution to \eqref{eq:weak equation with error} in $\Xi^\prime$, where the  function $h$ that appears in the definition of $G$ in \eqref{eq:G} depends on the modulus of continuity of $u_i$. Since the functions
$u_i$ are equicontinuous, this $h$ can be taken to be independent of $i$. As $u_{i,\varepsilon}$ are Lipschitz,
the Caccioppoli's estimate in Lemma \ref{lem:caccioppoli} implies the local boundedness of $Du_{i,\varepsilon}$ in $L^p_{\text{loc}}(\Xi^\prime)$. Taking the diagonal, that is, setting $v_i=u_{i,i^{-1}}$, we have by Lemma \ref{lem:strong conv} that $Dv_i$ converges for any $1<r<p$ in $L^r_{\text{loc}}(\Xi^\prime)$ up to a subsequence. An argument similar to the proof of Theorem \ref{thm:equivalence} then shows that the pointwise limit of $v_i$ is a local weak supersolution in $\Xi^\prime$ and this limit is $u$, as by equicontinuity $u_{i,\varepsilon}$ converges locally uniformly in $\Xi$ to $u_i$ as $\varepsilon \rightarrow 0$ with a speed independent of $i$.
\end{proof}

\section{Local Lipschitz continuity in space}

We prove local Lipschitz continuity of positive viscosity solutions
to (\ref{eq:trudinger}) in the space variable using the Ishii-Lions
method. This is because in the proof,
the solution is essentially frozen, and therefore the H{\"o}lder and Lipschitz
estimates produced by the method will depend on the modulus of continuity
of the solution.

Let $u:\Xi\rightarrow\mathbb{R}$, where $\Xi\subset\mathbb{R}^{N+1}$
is a set. We say that 
\[
\omega(s):=\sup_{(x,t),(y,t)\in\Xi,\left|x-y\right|\leq s}\left|u(x,t)-u(y,t)\right|
\]
is the \textit{optimal modulus of continuity} of $u$ in $\Xi$ in
space variable. We carefully observe that we do not assume Lipschitz continuity at this point, and the subindex $x$ is omitted in the notation of $\omega$. We also denote
\[
Q_{r}:=B_{r}(0)\times(-r,0]\quad\text{for }r>0
\]
and its parabolic boundary by $\partial_{\mathcal{P}} Q_r := (\overline B_r (0) \times \{-r\}) \cup (\partial B_r (0) \times (-r, 0]).$
\begin{thm}[Lipschitz continuity]
\label{thm:lipschitz theorem}Let $u$ be a uniformly continuous
viscosity solution to (\ref{eq:trudinger}) in $B_{2}\times(-2,0)$.
Suppose  that $0<m\leq u\leq M$. Then there exists $L>0$
such that
\[
\left|u(x_{0},t_{0})-u(y_{0},t_{0})\right|\leq L\left|x_{0}-y_{0}\right|\quad\text{for all }(x_{0},t_{0}),(y_{0},t_{0})\in Q_{1/2}
\]
The constant $L$ depends only on $N$, $p$, $m$ and the optimal modulus
of continuity of $u$ in $B_{2}\times(-2,0)$.
\end{thm}

The optimality of $\omega$ in Theorem \ref{thm:lipschitz theorem} is only for convenience: One could choose another modulus $\omega_2$ such that $\lim_{s\rightarrow 0} \omega_2(s) = 0$, repeat the proof using $\omega_2$, and obtain another Lipschitz constant $L_2$ that depends on $\omega_2$. However, analyzing the proof one would find that $L \leq L_2$ since $\omega \leq \omega_2$. Therefore, it is convenient to state the theorem using the optimal modulus, as it is uniquely defined. In addition, using H{\"o}lder results from the weak theory \cite{kuusisu12, kuusilsu12}, we get rid of the modulus dependence altogether.

The local Lipschitz continuity yields the H{\"o}lder continuity of viscosity solutions in time with exponent $1/2$, which we establish in Theorem \ref{thm:holder continuity in time}. Using the previous equivalence result Theorem~\ref{thm:equivalence}, we immediately have the same regularity for weak solutions. 
\begin{corollary}[Weak solutions are Lipschitz in space and $1/2$-H{\"o}lder in time]
\label{cor:lip}
Suppose that $0<m\leq u\leq M$. If $u$ is a weak (or viscosity) solution to (\ref{eq:trudinger}) in $B_{1}\times(-1,0)$, then $u$ is Lipschitz continuous in space and H{\"o}lder continuous in time in $Q_{1/2}$. That is, there is $C>0$ such that 
\begin{equation*}
    |u(x,t) - u(y,s)| \leq C(|x-y|+|t-s|^{1/2})
\end{equation*}
for all $(x,t), (y,s) \in Q_{1/2}$. The constant $C$ depends only on $N$, $p$, $m$ and $M$.
\end{corollary}

Since the equation (\ref{eq:trudinger}) is parabolic, it is natural
to establish the Lipschitz continuity of solution $u$ in $B_{2}\times(-2,0)$
up to the top of the cylinder $Q_{1/2}$. For this, we need the following lemma
that extends the solution property to the set $B_{2}\times\left\{ 0\right\} $.
For equations that are proper in the language of \cite{userguide},
a proof of such a result can be found in \cite[p209]{diehlFrizOberhauser14}.
However, (\ref{eq:trudinger}) is not proper and therefore we rely
on the equivalence result and the comparison principle of weak solutions
to prove the lemma. The technique resembles the one in \cite{juutinen01}.

\sloppy For $(x,t)\in Q_{r}$, we define the \textit{second-order subjet relative
to $Q_{r}$} by
\begin{align*}
\mathcal{P}_{Q_{r}}^{2,-}u(x,t):=\Big\{ & (\theta,\eta,X):u(y,s)\ge u(x,t)+(s-t)\theta+(x-y)\cdot\eta+\frac{1}{2}(x-y)^{\prime}X(x-y)\\
 & +o(\left|x-y\right|^{2}+\left|s-t\right|)\quad\text{as}\quad(y,s)\rightarrow(x,t),(y,s)\in Q_{r}\Big\}
\end{align*}
and $(\theta,\eta,X)$ is in $\overline{\mathcal{P}}_{Q_{r}}^{2,-}u(x,t)$
if there exists $(\theta_{i},\eta_{i},X_{i})\in\mathcal{P}_{Q_{r}}^{2,-}u(x_{i},t_{i})$
such that $(x_{i},t_{i},\theta_{i},\eta_{i},X_{i})$ converges to
$(x,t,\theta,\eta,X)$. The \textit{second-order superjet relative
to $Q_{r}$} is denoted by $\mathcal{P}_{Q_{r}}^{2,+}u(x,t)$ and
defined in the analogous way. Finally, note that these subjets coincide
with the usual definition within the interior of $Q_{r}$. Note that we regard $u$ as continuously extended to the top of the cylinder.

\begin{lem}
\label{lem:relative jets lemma} Let $u$ be a uniformly continuous
viscosity supersolution to (\ref{eq:trudinger}) in $B_{1}\times(-1,0)$
such that $u>m>0$. Suppose that $(\theta,\eta,X)\in\mathcal{\overline{P}}_{Q_{1}}^{2,-}u(x,0)$,
$\eta\not=0$, where $u(x,0)=\lim_{(y,s)\rightarrow(x,0),(y,s)\in Q}u(y,s)$.
Then
\[
(p-1)\theta-u^{2-p}(x,0)(\tr X+(p-2)\frac{\eta^{\prime}X\eta}{\left|\eta\right|^{2}})\geq0.
\]
The analogous result holds for viscosity subsolutions.
\end{lem}

\begin{proof}
The idea is that if the claim does not hold, then we can construct a $C^2$ function $\varphi$ that touches $u$ from below at the top of the cylinder, and $\varphi$ has to be a strict subsolution near the point of touching. By lifting this function a little, we can get a contradiction since $u$ satisfies the comparison principle.

To be more precise, by uniform continuity of $u$, it is enough to consider $(\theta, \eta, X) \in \mathcal{P}_{Q_1}^{2, -} u(x, 0)$ instead of the closure. Thriving for a contradiction, suppose that there exists $(\theta,\eta,X)\in \mathcal{P}_{Q_{1}}^{2,-}u(x,0),\eta\not=0$,
such that 
\[
(p-1)\theta-u^{2-p}(x,0)(\tr X+(p-2)\frac{\eta^{\prime}X\eta}{\left|\eta\right|^{2}})<-\varepsilon
\]
for some $\varepsilon>0$. Then there is $\varphi\in C^{2}(Q_{1})$
such that $(\partial_{t}\varphi(x,0),D\varphi(x,0),D^{2}\varphi(x,0))=(\theta,\eta,X)$,
$\varphi(x,0)=u(x,0)$, $\varphi(y,s)\l u(y,s)$ when $(y,s)\in Q_{1}\setminus\{(x,0)\}$,
and
\[
(p-1)\partial_{t}\varphi(x,0)-\varphi^{2-p}(x,0)\Delta_{p}\varphi(x,0)\l-\varepsilon.
\]
Since $\varphi$ is $C^{2}$, it follows from above that there exists small
$\delta>0$ such that 
\[
\partial_{t}\varphi^{p-1}(y,s)-\Delta_{p}\varphi(y,s)<-\varepsilon/2\quad\text{in }B_{\delta}\times(-\delta, 0)
\]
and also $\varphi>0$, $D\varphi\not=0$ in $Q_{\delta}(x,0)$ (since $D\varphi(x,0) = \eta \not= 0$). From this
it follows that $\varphi$ is, in particular, a weak subsolution to
(\ref{eq:trudinger}) in $B_{\delta}\times(-\delta,0)$. Then so is
\[
\tilde{\varphi}(y,s):=h\varphi(y,s),\quad\text{where }h:=\inf_{(y,s)\in\partial_{\mathcal{P}}Q_{\delta}(x,0)}\frac{u(y,s)}{\varphi(y,s)}>1.
\]
On the other hand, by Theorem \ref{thm:visc is weak} $u$ is a local
weak supersolution to (\ref{eq:trudinger}) in $B_{\delta}\times(-\delta,0)$.
Since by construction of $\tilde{\varphi}$ we have $\tilde{\varphi}\leq u$
on $\partial_{\mathcal{P}}Q_{\delta}$, it follows from the comparison
principle \cite[Corollary 1]{lindgrenl22} and continuity
that $\tilde{\varphi}\leq u$ in $Q_{\delta}$, which leads to a contradiction
since $\tilde{\varphi}(x,0)=h\varphi(x,0)=hu(x,0)>u(x,0).$
\end{proof}
To prove Theorem \ref{thm:lipschitz theorem}, we first prove the
following technical lemma, which we will apply twice: first to obtain
H{\"o}lder continuity, and then Lipschitz continuity. This lemma encapsulates the use of Theorem on Sums, and thus makes the full proof cleaner, since then we avoid repeating that part in both the H{\"o}lder and and Lipschitz proofs. 

In the following discussion, we drop the localization terms for expository reasons similar to \cite{anttilamp}.
The idea in the Ishii-Lions method when aiming at 
\begin{align*}
    u(x,t)-u(y,t)\le L\varphi(\left|x-y\right|),
\end{align*}
where  $\varphi$ reflects the desired regularity (H\"older or Lipschitz as detailed below in (\ref{eq:explicit})), is  
thriving for a contradiction to assume that there is a maximum point such that
\begin{align*}
    u(x_0,t_0)-u(y_0,t_0)- L\varphi(\left|x_0-y_0\right|)>c>0.
\end{align*}
Then using the maximum point property, we have information about the derivatives and second derivatives, which we can use to contradict the opposite information coming from the equation in the same way as in the usual uniqueness machinery for viscosity solutions. 

The proof needs to be done in two stages since in the H\"older proof at the maximum point, we can utilize the larger negative curvature of $\varphi$, in other words larger negativite impact of the second derivative to a suitable direction, as done in (\ref{eq:ishii-lions-holder}) and in the computation  after the estimate. 
In the Lipscitz proof the second derivative at the maximum point is not as negative but there (see (\ref{eq:ishii-lions-Lip})) the positive (bad) term depends on the modulus of continuity of the solution, and this is sufficiently small by the previous H\"older result. 

Since the core of the argument is the same in both the stages it is stated separately in the following lemma. It might be instructive to observe that in the proofs of H\"older and Lipschitz continuity, the lemma will be utilized respectively with the functions
\begin{equation}
\label{eq:explicit}
    \varphi(s) := s^{\alpha}\quad \text{and} \quad \varphi (s) := s - \kappa s^\beta,
\end{equation}
where $\alpha, \kappa$ and $\beta$ are suitable constants.
\begin{lem}
\label{lem:Ishii-Lions lemma} Suppose that $u$ is a uniformly continuous
viscosity solution to (\ref{eq:trudinger}) in $B_{1}\times(-1,0)$
and $0<m\leq u\leq M$. Let $(x_{0},t_{0}),(y_{0},t_{0})\in Q_{1/2}$
and define the function
\begin{equation}
\Psi(x,y,t):=u(x,t)-u(y,t)-L\varphi(\left|x-y\right|)-\frac{K}{2}\left|x-x_{0}\right|^{2}-\frac{K}{2}\left|y-y_{0}\right|^{2}-\frac{K}{2}\left|t-t_{0}\right|^{2},\label{eq:lemma blah}
\end{equation}
where $K:=8\osc_{Q_{1}}u$, and $\varphi:[0,2]\rightarrow[0,\infty)$
is a $C^{2}$-function such that
\begin{equation}
\varphi(0)=0,\quad\ensuremath{\left|\varphi^{\prime\prime}(s)\right|<\frac{\varphi^{\prime}(s)}{s}}\quad\text{and}\quad\varphi^{\prime\prime}<0<c_{\varphi}<\varphi^{\prime}\text{\ensuremath{\quad\text{for }\text{some }c_{\varphi}>0.}}\label{eq:Ishii-Lions lemma cnd}
\end{equation}
Then, if $L>L^{\prime}$ for some $L^{\prime}$ that depends only
on $c_{\varphi}$ and $\osc_{Q_{1}}u$, the following holds: If $\Psi$
has a positive maximum at $(\hat{x},\hat{y},\hat{t})\in\overline{B}_{1}\times\overline{B}_{1}\times[-1,0]$,
we have 
\[
-K\leq(L\varphi^{\prime}(\left|z\right|))^{p-2}\left(C^{-2}L\varphi^{\prime\prime}(\left|z\right|)+L\frac{\omega(\left|z\right|)\varphi^{\prime}(\left|z\right|)}{\left|z\right|}+\sqrt{K}\frac{\omega^{1/2}(\left|z\right|)}{\left|z\right|}+K\right),
\]
where $\left|z\right|=\left|\hat{x}-\hat{y}\right|$ and $\omega$
is the optimal modulus of continuity of $u$ in $B_{1}\times(-1,0)$
in space variable. Here $C$ depends only on $N,p,m$ and $M$. Finally,
we have the estimate
\begin{equation}
\varphi^{\prime}(\left|z\right|)\leq\frac{\omega(\left|z\right|)}{L\left|z\right|}.\label{eq:lemma derivative est}
\end{equation}
\end{lem}

\begin{proof}
First, we must have $\left|z\right|\not=0$ since otherwise the maximum
at $(\hat{x},\hat{y},\hat{t})$ would be non-positive. Secondly, we
have
\begin{align}
0 & <\left|u(\hat{x},\hat{t})-u(\hat{y},\hat{t})\right|-L\varphi(\left|\hat{x}-\hat{y}\right|)-\frac{K}{2}\left|\hat{x}-x_{0}\right|^{2}-\frac{K}{2}\left|\hat{y}-y_{0}\right|^{2}-\frac{K}{2}\left|\hat{t}-t_{0}\right|^{2}\label{eq:lemma whatever}
\end{align}
so that 
\begin{equation}
\left|\hat{t}-t_{0}\right|,\left|\hat{x}-x_{0}\right|,\left|\hat{y}-y_{0}\right|\leq\sqrt{\frac{2}{K}\left|u(\hat{x},\hat{t})-u(\hat{y},\hat{t})\right|}\leq\sqrt{\frac{1}{4}}=\frac{1}{2}.\label{eq:lemma est 11}
\end{equation}
Since $x_{0},y_{0}\in B_{1/2}$ and $t_{0}\in(-1/2,0]$, this implies
$\hat{x},\hat{y}\in B_{1}$ and $\hat{t}\in(-1,0]$, which means that
the maximum point $(\hat{x},\hat{y},\hat{t})$ is in $B_{1}\times B_{1}\times(-1,0]$.

By the definition of $\omega$ and the uniform continuity of $u$, we have
\begin{equation}
\left|u(x,s)-u(y,s)\right|\leq\omega(\left|x-y\right|)\quad\text{for all }(x,s),(y,s)\in B_{1}\times(0,1],\label{eq:3-1}
\end{equation}
Therefore, (\ref{eq:lemma est 11}) implies that 
\begin{align}
K\left|\hat{t}-t_{0}\right|,K\left|\hat{x}-x_{0}\right|,K\left|\hat{y}-y_{0}\right|\leq & \sqrt{K}\omega^{1/2}(\left|\hat{x}-\hat{y}\right|). \label{eq:distest}
\end{align}
We also have by concavity of $\varphi$ and (\ref{eq:lemma whatever})
that $\left|z\right|\varphi^{\prime}(\left|z\right|)\leq\int_{0}^{\left|z\right|}\varphi^{\prime}(s)\d s\leq\frac{\omega(\left|z\right|)}{L}$
so that (\ref{eq:lemma derivative est}) holds. 

Since $\hat{x}\not=\hat{y}$, the function $\phi(x,y)\mapsto\varphi(\left|x-y\right|)$
is $C^{2}$ in a neighborhood of $(\hat{x},\hat{y})$. Therefore,
we may invoke the parabolic Theorem on Sums \cite[Theorem 9]{diehlFrizOberhauser14}
(see also \cite[Theorem 8.3]{userguide}). For any $\mu>0$, there
exist matrices $X,Y\in S^{N}$ and $b_{1},b_{2}\in\mathbb{R}$ such
that 
\[
b_{1}+b_{2}\geq\partial_{t}(L\varphi(\left|x-y\right|))(\hat{x},\hat{y})=0
\]
and
\begin{align*}
(b_{1},D_{x}(L\varphi(\left|x-y\right|))(\hat{x},\hat{y}),X) & \in\overline{\mathcal{P}}_{Q_{1}}^{2,+}(u-\frac{K}{2}\left|x-x_{0}\right|^{2}-\frac{K}{2}\left|t-t_{0}\right|^{2})(\hat{x},\hat{t}),\\
(-b_{2},D_{y}(L\varphi(\left|x-y\right|))(\hat{x},\hat{y}),Y) & \in\overline{\mathcal{P}}_{Q_{1}}^{2,-}(u+\frac{K}{2}\left|y-y_{0}\right|^{2})(\hat{y},\hat{t}).
\end{align*}
Denoting $z:=\hat{x}-\hat{y}$ and
\begin{align}
\eta_{1} & :=L\varphi^{\prime}(\left|z\right|)\frac{z}{\left|z\right|}+K(\hat{x}-x_{0}),\nonumber \\ 
\eta_{2} & :=L\varphi^{\prime}(\left|z\right|)\frac{z}{\left|z\right|}-K(\hat{y}-y_{0}), \label{eq:eta2}\\
\theta_{1} & :=b_{1}+K(\hat{t}-t_{0}),\nonumber \\
\theta_{2} & :=-b_{2} \nonumber,
\end{align}
these can be written as
\begin{align*}
(\theta_{1},\eta_{1},X+KI)\in\overline{\mathcal{P}}_{Q_{1}}^{2,+}u(\hat{x},\hat{t}),\quad & (\theta_{2},\eta_{2},Y-KI)\in\overline{\mathcal{P}}_{Q_{1}}^{2,-}u(\hat{y},\hat{t}).
\end{align*}
Assuming $L$ is large enough depending on $\sqrt{K}$, $\osc_{Q_{1}}u$
and $c_{\varphi}$, and using \eqref{eq:distest}, \eqref{eq:eta2}, we have 
\begin{align}
\left|\eta_{1}\right|,\left|\eta_{2}\right| & \leq L\varphi^{\prime}(\left|z\right|)+\sqrt{K}\omega^{1/2}(\left|\hat{x}-\hat{y}\right|)\leq L\varphi^{\prime}(\left|z\right|)+Lc_{\varphi}\leq2L\varphi^{\prime}(\left|z\right|)\nonumber \\
\left|\eta_{1}\right|,\left|\eta_{2}\right| & \geq\frac{1}{2}L\varphi^{\prime}(\left|z\right|)+\frac{1}{2}Lc_{\varphi}-\sqrt{K}\omega^{1/2}(\left|\hat{x}-\hat{y}\right|)\geq\frac{1}{2}L\varphi^{\prime}(\left|z\right|)\geq1,\label{eq:lemma gradient estimate}
\end{align}
Furthermore, by Theorem on Sums, we have 
\begin{align}
-(\mu+4\left\Vert B\right\Vert ) & I\leq\begin{pmatrix}X & 0\\
0 & -Y
\end{pmatrix}\leq\begin{pmatrix}B & -B\\
-B & B
\end{pmatrix}+\frac{2}{\mu}\begin{pmatrix}B^{2} & -B^{2}\\
-B^{2} & B^{2}
\end{pmatrix},\label{eq:lemma matrix ineq}
\end{align}
where
\begin{align*}
B & =L\varphi^{\prime\prime}(\left|z\right|)\frac{z\otimes z}{\left|z\right|^{2}}+L\frac{\varphi^{\prime}(\left|z\right|)}{\left|z\right|}\left(I-\frac{z\otimes z}{\left|z\right|^{2}}\right),\\
B^{2}=BB & =L^{2}\left|\varphi^{\prime\prime}(\left|z\right|)\right|^{2}\frac{z\otimes z}{\left|z\right|^{2}}+L^{2}\frac{\left|\varphi^{\prime}(\left|z\right|)\right|^{2}}{\left|z\right|^{2}}\left(I-\frac{z\otimes z}{\left|z\right|^{2}}\right).
\end{align*}
By assumptions on $\varphi$, we have
\[
\left|\varphi^{\prime\prime}(\left|z\right|)\right|\leq\varphi^{\prime}(\left|z\right|)/\left|z\right|\quad\text{and}\quad\varphi^{\prime\prime}(\left|z\right|)<0<\varphi^{\prime}(\left|z\right|).
\]
Thus we deduce that 
\begin{equation}
\left\Vert B\right\Vert \leq\frac{L\varphi^{\prime}(\left|z\right|)}{\left|z\right|}\quad\text{and}\quad\left\Vert B^{2}\right\Vert \leq L^{2}\frac{(\varphi^{\prime}(\left|z\right|))^{2}}{\left|z\right|^{2}}.\label{eq:lemma B est}
\end{equation}
Moreover, we choose
\[
\mu:=\frac{2L\varphi^{\prime}(\left|z\right|)}{\left|z\right|},
\]
so that using the explicit formulae of $B$ and $B^2$, and the fact that the terms with $\varphi^\prime$ vanish, we obtain
\begin{equation}
\left\langle B\frac{z}{\left|z\right|},\frac{z}{\left|z\right|}\right\rangle +\frac{2}{\mu}\left\langle B^{2}\frac{z}{\left|z\right|},\frac{z}{\left|z\right|}\right\rangle \leq\frac{L}{2}\varphi^{\prime\prime}(\left|z\right|).\label{eq:lemma est 33}
\end{equation}
Denote below
\begin{align*}
F((z,s),\zeta,Z) & :=u^{2-p}(z,s)\left|\zeta\right|^{p-2}(\tr Z+(p-2)\frac{\zeta^{\prime}Z\zeta}{\left|\zeta\right|^{2}}).
\end{align*}
Since $u$ is a viscosity solution to (\ref{eq:trudinger}), we have
(using also Lemma \ref{lem:relative jets lemma} if $\hat{t}=0$)
\[
\theta_{1}\leq F((\hat{x},\hat{t}),\eta_{1},X+KI)\quad\text{and\ensuremath{\quad}}\theta_{2}\geq F((\hat{y},\hat{t}),\eta_{2},Y-KI).
\]
We denote $\tilde{\eta}:=L\varphi^{\prime}(\left|z\right|)z/\left|z\right|$.
Then we subtract the equations and use that $\theta_{1}-\theta_{2}=b_{1}+b_{2}+K(\hat{t}-t_{0})\geq K(\hat{t}-t_{0})$ as well as add and subtract terms to obtain 
\begin{align*}
K(\hat{t}-t_{0}) & \leq F((\hat{x},\hat{t}),\eta_{1},X+KI)-F((\hat{y},\hat{t}),\eta_{2},Y-KI)\\
 & =F((\hat{x},\hat{t}),\eta_{1},X)-F((\hat{y},\hat{t}),\eta_{2},Y)\\
 & \ \ \ +F((\hat{x},\hat{t}),\eta_{1},X+KI)-F((\hat{x},\hat{t}),\eta_{1},X)-F((\hat{y},\hat{t}),\eta_{2},Y-KI)+F((\hat{y},\hat{t}),\eta_{2},Y)\\
 & =F((\hat{x},\hat{t}),\tilde{\eta},X)-F((\hat{y},\hat{t}),\tilde{\eta},Y)\\
 & \ \ \ +F((\hat{x},\hat{t}),\eta_{1},X)-F((\hat{x},\hat{t}),\tilde{\eta},X)+F((\hat{y},\hat{t}),\tilde{\eta},Y)-F((\hat{y},\hat{t}),\eta_{2},Y)\\
 & \ \ \ +F((\hat{x},\hat{t}),\eta_{1},X+KI)-F((\hat{x},\hat{t}),\eta_{1},X)-F((\hat{y},\hat{t}),\eta_{2},Y-KI)+F((\hat{y},\hat{t}),\eta_{2},Y)\\
 & =:T_{1}+T_{2}+T_{3}.
\end{align*}

\textbf{Estimate of $T_{1}$:} By the definition of $\tilde{\eta}$,
the matrix inequality (\ref{eq:lemma matrix ineq}), and estimate (\ref{eq:lemma est 33}),
we have
\begin{align*}
\frac{\tilde{\eta}^{\prime}}{\left|\tilde{\eta}\right|}(X-Y)\frac{\tilde{\eta}}{\left|\tilde{\eta}\right|}=\frac{z^{\prime}}{\left|z\right|}(X-Y)\frac{z}{\left|z\right|} & =\begin{pmatrix}z/\left|z\right|\\
-z/\left|z\right|
\end{pmatrix}^{\prime}\begin{pmatrix}X & 0\\
0 & -Y
\end{pmatrix}\begin{pmatrix}z/\left|z\right|\\
-z/\left|z\right|
\end{pmatrix}\\
 & \leq\begin{pmatrix}z/\left|z\right|\\
-z/\left|z\right|
\end{pmatrix}^{\prime}\left(\begin{pmatrix}B & -B\\
-B & B
\end{pmatrix}+\frac{2}{\mu}\begin{pmatrix}B^{2} & -B^{2}\\
-B^{2} & B^{2}
\end{pmatrix}\right)\begin{pmatrix}z/\left|z\right|\\
-z/\left|z\right|
\end{pmatrix}\\
 & =2\left\langle B\frac{z}{\left|z\right|},\frac{z}{\left|z\right|}\right\rangle +\frac{4}{\mu}\left\langle B^{2}\frac{z}{\left|z\right|},\frac{z}{\left|z\right|}\right\rangle \\
 & \leq L\varphi^{\prime\prime}(\left|z\right|).
\end{align*}
Furthermore, (\ref{eq:lemma matrix ineq}) readily implies $X-Y\leq0$ by multiplying \eqref{eq:lemma matrix ineq} by $(\xi, \xi)$ left and right.
In particular, this implies $\tr(X-Y)\leq0$. Moreover, (\ref{eq:lemma matrix ineq})
also gives that
\begin{equation}
\left\Vert X\right\Vert ,\left\Vert Y\right\Vert \leq\left\Vert B\right\Vert +\frac{2}{\mu}\left\Vert B\right\Vert ^{2}\leq C\frac{L\varphi^{\prime}(\left|z\right|)}{\left|z\right|}.\label{eq:matrix norms}
\end{equation}
Using the previous two estimates and that $\varphi^{\prime\prime}<0$, it follows that
\begin{align*}
T_{1} & =F((\hat{x},\hat{t}),\tilde{\eta},X)-F((\hat{y},\hat{t}),\tilde{\eta},Y)\\
 & =u(\hat{x},\hat{t})^{2-p}\left|\tilde{\eta}\right|^{p-2}\Big(\tr X+(p-2)\frac{\tilde{\eta}^{\prime}X\tilde{\eta}}{\left|\tilde{\eta}\right|^{2}}\Big)-u(\hat{y},\hat{t})^{2-p}\left|\tilde{\eta}\right|^{p-2}\Big(\tr Y+(p-2)\frac{\tilde{\eta}^\prime Y\tilde{\eta}}{\left|\tilde{\eta}\right|^{2}}\Big)\\
 & =\big(u^{2-p}(\hat{x},\hat{t})-u^{2-p}(\hat{y},\hat{t})\big)\left|\tilde{\eta}\right|^{p-2}\Big(\tr X+(p-2)\frac{\tilde{\eta}^{\prime}X\tilde{\eta}}{\left|\tilde{\eta}\right|^{2}}\Big)\\
 & \ \ \ +u(\hat{y},\hat{t})^{2-p}\left|\tilde{\eta}\right|^{p-2}\Big(\tr(X-Y)+(p-2)\frac{\tilde{\eta}^\prime(X-Y)\tilde{\eta}}{\left|\tilde{\eta}\right|^{2}}\Big)\\
 & \leq C(N,p,m,M)\left|u(\hat{x},\hat{t})-u(\hat{y},\hat{t})\right|\left|\tilde{\eta}\right|^{p-2}\left\Vert X\right\Vert +C(p)u(\hat{y},\hat{t})^{2-p}\left|\tilde{\eta}\right|^{p-2}\frac{\tilde{\eta}^\prime(X-Y)\tilde{\eta}}{\left|\tilde{\eta}\right|^{2}}\\
 & \leq C(N,p,m,M)\omega(\left|z\right|)\left|\tilde{\eta}\right|^{p-2}L\frac{\varphi^{\prime}(\left|z\right|)}{\left|z\right|}+C^{-1}(N,p,m,M)L\left|\tilde{\eta}\right|^{p-2}\varphi^{\prime\prime}(\left|z\right|),
\end{align*}
for some constant $C(N,p,m,M)\geq1$.

\textbf{Estimate of $T_{2}$:} First we show the following algebraic
inequality
\begin{align}
 & \left|\left|\xi_{1}\right|^{p-2}(\tr Z+(p-2)\left|\xi_{1}\right|^{-2}\xi_{1}^{\prime}Z\xi_{1})-\left|\xi_{2}\right|^{p-2}(\tr Z+(p-2)\left|\xi_{2}\right|^{-2}\xi_{2}^{\prime}Z\xi_{2})\right|\nonumber \\
 & \ \ \ \leq C(N,p)\max(\left|\xi_{1}\right|^{p-3},\left|\xi_{2}\right|^{p-3})\left|\xi_{1}-\xi_{2}\right|\left\Vert Z\right\Vert \label{eq:ineq}
\end{align}
for all $\xi_{1},\xi_{2}\in\mathbb{R}^{N}\setminus\left\{ 0\right\} $
and $Z\in S(N)$. To this end, we observe that by \cite{theobald75}
\begin{align*}
\left|\xi_{1}^{\prime}Z\xi_{1}-\xi_{2}Z\xi_{2}\right| & =\left|\tr((\xi_{1}\otimes\xi_{1}-\xi_{2}\otimes\xi_{2})Z)\right|\\
 & \leq N\left\Vert \xi_{1}\otimes\xi_{1}-\xi_{2}\otimes\xi_{2}\right\Vert \left\Vert Z\right\Vert \\
 & =N\left\Vert (\xi_{1}-\xi_{2})\otimes\xi_{1}-\xi_{2}\otimes(\xi_{2}-\xi_{1})\right\Vert \left\Vert Z\right\Vert \\
 & \leq N(\left|\xi_{1}\right|+\left|\xi_{2}\right|)\left\Vert \xi_{1}-\xi_{2}\right\Vert \left\Vert Z\right\Vert .
\end{align*}
We also use the elementary inequality
\[
\left|a^{q}-b^{q}\right|\leq \abs q\max(a^{q-1},b^{q-1})\left|a-b\right|,\quad a,b>0\quad\text{and}\quad q\in\mathbb{R}.
\]
Furthermore, may assume that $\left|\xi_{2}\right|\leq\left|\xi_{1}\right|$
by changing notation, if necessary. Using these, we estimate
\begin{align*}
 & \left|\left|\xi_{1}\right|^{p-2}-\left|\xi_{2}\right|^{p-2}\right|\left|\tr Z\right|+(p-2)\left|\xi_{1}\right|^{p-4}\left|\xi_{1}^{\prime}Z\xi_{1}-\xi_{2}^{\prime}Z\xi_{2}\right|+(p-2)\left|\left|\xi_{1}\right|^{p-4}-\left|\xi_{2}\right|^{p-4}\right|\left|\xi_{2}^{\prime}Z\xi_{2}\right|\\
 & \leq N\left|p-2\right|\max(\left|\xi_{1}\right|^{p-3},\left|\xi_{2}\right|^{p-3})\left|\xi_{1}-\xi_{2}\right|\left\Vert Z\right\Vert +N\left|p-2\right|\left|\xi_{1}\right|^{p-4}(\left|\xi_{1}\right|+\left|\xi_{2}\right|)\left\Vert \xi_{1}-\xi_{2}\right\Vert \left\Vert Z\right\Vert \\
 & \ \ \ +\left|p-2\right|\left|p-4\right|\max(\left|\xi_{1}\right|^{p-5},\left|\xi_{2}\right|^{p-5})\left|\xi_{1}-\xi_{2}\right|\left|\xi_{2}\right|^{2}\left\Vert Z\right\Vert \\
 & \leq C(N,p)\left|p-2\right|\max(\left|\xi_{1}\right|^{p-3},\left|\xi_{2}\right|^{p-3})\left|\xi_1-\xi_{2}\right|\left\Vert Z\right\Vert ,
\end{align*}
which proves the algebraic inequality (\ref{eq:ineq}). Observe now that
\begin{equation}
\left|\eta_{1}-\tilde{\eta}\right|+\left|\eta_{2}-\tilde{\eta}\right|=K(\left|x_{0}-\hat{x}\right|+\left|y_{0}-\hat{y}\right|)\leq2\sqrt{K}\omega^{1/2}(\left|z\right|).\label{eq:lemma gradient tilde est}
\end{equation}
Therefore, using (\ref{eq:ineq}), we obtain
\begin{align*}
T_{2} & =F((\hat{x},\hat{t}),\eta_{1},X)-F((\hat{x},\hat{t}),\tilde{\eta},X)+F((\hat{y},\hat{t}),\tilde{\eta},Y)-F((\hat{y},\hat{t}),\eta_{2},Y)\\
 & =u(\hat{x},\hat{t})^{2-p}|\eta_1|^{p-2}\tr(X-(p-2)\frac{\eta_{1}^{\prime}X\eta_{1}}{|\eta_{1}|^2})-u(\hat{x},\hat{t})^{2-p}|\tilde{\eta}|^{p-2}\tr(X-(p-2)\frac{\tilde{\eta}^{\prime}X\tilde{\eta}}{|\tilde{\eta}|^{2}})\\
 & \ \ \ +u(\hat{y},\hat{t})^{2-p}|\tilde{\eta}|^{p-2}\tr(Y-(p-2)\frac{\tilde{\eta}^{\prime}Y\tilde{\eta}}{\left|\tilde{\eta}\right|^{2}})-u(\hat{y},\hat{t})^{2-p}|\eta_2|^{p-2}\tr(Y-(p-2)\frac{\eta_{2}^{\prime}Y\eta_{2}}{\left|\eta_{2}\right|^{2}})\\
 & \leq C(N,p,m,M)\max(\left|\eta_{1}\right|^{p-3},\left|\eta_{2}\right|^{p-3},\left|\tilde{\eta}\right|^{p-3})(\left\Vert X\right\Vert +\left\Vert Y\right\Vert )(\left|\eta_{1}-\tilde{\eta}\right|+\left|\eta_{2}-\tilde{\eta}\right|)\\
 & \leq C(N,p,m,M)\left|\tilde{\eta}\right|^{p-3}\sqrt{K}\omega^{1/2}(\left|z\right|)L\frac{\varphi^{\prime}(\left|z\right|)}{\left|z\right|}
\end{align*}
where in the last estimate we used (\ref{eq:matrix norms}), (\ref{eq:lemma gradient tilde est})
and that (\ref{eq:lemma gradient estimate}) implies $\left|\eta_{1}\right|,\left|\eta_{2}\right|\in[\frac{1}{2}\left|\tilde{\eta}\right|,2\left|\tilde{\eta}\right|]$.

\textbf{Estimate of $T_{3}$:} Since $F$ is quasilinear, we have
\begin{align*}
T_{3} & =F((\hat{x},\hat{t}),\eta_{1},X+KI)-F((\hat{x},\hat{t}),\eta_{1},X)-F((\hat{y},\hat{t}),\eta_{2},Y-KI)+F((\hat{y},\hat{t}),\eta_{2},Y)\\
 & =F((\hat{x},\hat{t}),\eta_{1},KI)+F((\hat{y},\hat{t}),\eta_{2},KI)\\
 & \leq C(N,p)K(\left|\eta_{1}\right|^{p-2}+\left|\eta_{2}\right|^{p-2})\\
 & \leq C(N,p)K\left|\tilde{\eta}\right|^{p-2},
\end{align*}
where we also used (\ref{eq:lemma gradient estimate}) and the definition
of $\tilde{\eta}$.

Combining the estimates of $T_{1},T_{2},T_{3}$ and using (\ref{eq:lemma gradient estimate})
we have for some $C(N,p,m,M)\geq1$ that

\begin{align*}
 & -K\\
 & \ \leq C^{-1}L\left|\tilde{\eta}\right|^{p-2}\varphi^{\prime\prime}(\left|z\right|)+C\omega(\left|z\right|)\left|\tilde{\eta}\right|^{p-2}L\frac{\varphi^{\prime}(\left|z\right|)}{\left|z\right|}+\left|\tilde{\eta}\right|^{p-3}\sqrt{K}\omega^{1/2}(\left|z\right|)L\frac{\varphi^{\prime}(\left|z\right|)}{\left|z\right|}+CK\left|\tilde{\eta}\right|^{p-2}\\
 & \ =C(L\varphi^{\prime}(\left|z\right|))^{p-2}\left(C^{-2}L\varphi^{\prime\prime}(\left|z\right|)+L\frac{\omega(\left|z\right|)\varphi^{\prime}(\left|z\right|)}{\left|z\right|}+\sqrt{K}\frac{\omega^{1/2}(\left|z\right|)}{\left|z\right|}+K\right),
\end{align*}
where we also used the defintion of $\tilde{\eta}$. This yields the
desired inequality.
\end{proof}

The next lemma gives $\alpha$-H\"older continuity for any $\alpha \in (0,1)$.

\begin{lem}[H{\"o}lder continuity in space]
 \label{lem:H=0000F6lder in space} Let $u$ be a uniformly continuous
viscosity solution to (\ref{eq:trudinger}) in $B_{2}\times(-2,0)$.
Suppose that $0<m\leq u\leq M$. Then for any $\alpha \in (0,1)$, there exists a constant
$C_{H}>0$ such that
\[
\left|u(x_{0},t_{0})-u(y_{0},t_{0})\right|\leq C_{H}\left|x_{0}-y_{0}\right|^{\alpha}\quad\text{for all }(x_{0},t_{0}),(y_{0},t_{0})\in Q_{1}.
\]
The constant $C_{H}$
depends only on $N,p,m,\alpha$ and the optimal modulus of continuity of
$u$ in $B_{2}\times(-2,0)$.
\end{lem}

\begin{proof}
Following the notation in Lemma \ref{lem:Ishii-Lions lemma}, we fix
$K:=8\osc_{Q_{1}}u$. We also fix $(x_{0},t_{0}),(y_{0},t_{0})\in Q_{1/2}$
and consider the auxiliary function
\[
\Psi(x,y,t):=u(x,t)-u(y,t)-C_{H}\varphi(\left|x-y\right|)-\frac{K}{2}\left|x-x_{0}\right|^{2}-\frac{K}{2}\left|y-y_{0}\right|^{2}-\frac{K}{2}\left|t-t_{0}\right|^{2}
\]
for some large $C_{H}>0$ to be selected later. We define $\varphi:[0,2]\rightarrow[0,\infty)$
by $\varphi(s):=s^{\alpha}$, where $\alpha\in(0,1)$. We have
\[
\varphi^{\prime\prime}(s)=\alpha(\alpha-1)s^{\alpha-2}<0<\alpha2^{\alpha-1}<\varphi^{\prime}(s)\quad\text{and\ensuremath{\quad\left|\varphi^{\prime\prime}(s)\right|\leq\frac{\varphi^{\prime}(s)}{s}}}
\]
so that condition (\ref{eq:Ishii-Lions lemma cnd}) in Lemma \ref{lem:Ishii-Lions lemma}
holds. It now suffices to prove that $\Psi$ is non-positive, so suppose
on the contrary that $\Psi$ has a positive maximum at $(\hat{x},\hat{y},\hat{t})\in\overline{B}_{1}\times\overline{B}_{1}\times[-1,1]$.
Provided that $C_{H}\geq C(\alpha,\osc_{Q_{1}}u)$, which we now assume,
we may apply Lemma \ref{lem:Ishii-Lions lemma}. There exists $C_{0}$,
depending only on $N,p,m$, and $M$, such that
\begin{align}
\label{eq:ishii-lions-holder}
 & -K\nonumber \\
 & \leq(C_{H}\varphi^{\prime}(\left|z\right|))^{p-2}\left(C_{0}^{-2}C_{H}\varphi^{\prime\prime}(\left|z\right|)+C_{H}\omega(\left|z\right|)\frac{\varphi^{\prime}(\left|z\right|)}{\left|z\right|}+\sqrt{K}\frac{\omega^{1/2}(\left|z\right|)}{\left|z\right|}+K\right)\\
 & =(\varphi^{\prime}(\left|z\right|)C_{H})^{p-2}\left(-C_{0}^{-2}C_{H}\alpha\left|\alpha-1\right|\left|z\right|^{\alpha-2}+\alpha C_{H}\omega(\left|z\right|)\left|z\right|^{\alpha-2}+\sqrt{K}\frac{\omega^{1/2}(\left|z\right|)}{\left|z\right|}+K\right)\nonumber,
\end{align}
where $\omega$ is the optimal modulus of continuity of $u$ in $Q_{1}$
in space variable and $z=\hat{x}-\hat{y}$ for some $(\hat{x},\hat{t}),(\hat{y},\hat{t})\in B_{1}\times(-1,0)$.
By the estimate (\ref{eq:lemma derivative est}) we have 
\[
\omega(\left|z\right|)\leq\omega\left(\frac{\omega(\left|z\right|)}{\varphi^{\prime}(\left|z\right|)C_{H}}\right)\leq\omega\left(\frac{\osc_{Q}u}{\alpha2^{\alpha-1}C_{H}}\right).
\]
Then by assuming that $C_{H}$ is even larger, depending on $\osc_{Q_{1}}u$,
$\alpha$, $C_{0}$, and $\omega$, we may ensure that 
\[
\omega(\left|z\right|)\leq C_{0}^{-2}\frac{\left|\alpha-1\right|}{2}.
\]
Therefore, continuing the estimate, we have
\begin{alignat*}{1}
 & -K\\
 & \ \leq(\varphi^{\prime}(\left|z\right|)C_{H})^{p-2}\left(-C_{0}^{-2}C_{H}\alpha\left|\alpha-1\right|\left|z\right|^{\alpha-2}+C_{0}^{-2}\alpha\left|\alpha-1\right|C_{H}\left|z\right|^{\alpha-2}+\sqrt{K}\frac{\omega^{1/2}(\left|z\right|)}{\left|z\right|}+K\right)\\
 & \ =(\varphi^{\prime}(\left|z\right|)C_{H})^{p-2}\left(-\frac{C_{0}^{-2}}{2}C_{H}\alpha\left|\alpha-1\right|\left|z\right|^{\alpha-2}+\sqrt{K}\frac{(\osc_{Q_1}u)^{1/2}}{\left|z\right|}+K\right)\\
 & \ \leq(\varphi^{\prime}(\left|z\right|)C_{H})^{p-2}\left|z\right|^{\alpha-2}\left(-\frac{C_{0}^{-2}\alpha\left|\alpha-1\right|}{2}C_{H}+C\sqrt{K}(\osc_{Q_1}u)^{1/2}+CK\right),
\end{alignat*}
where in the last estimate we used that $\left|z\right|^{-1}\leq C\left|z\right|^{\alpha-2}$
since $\left|z\right|<2$ and $\alpha-2\in(-2,-1)$. Then we take
$C_{H}$ so large that $C_{H}\geq C\frac{4C_{0}^{2}}{\alpha(\alpha-1)}(\sqrt{K}(\osc_{Q} u)^{1/2}+K)$
so that by the above display (using also (\ref{eq:lemma derivative est})
to estimate $\left|z\right|^{\alpha-2}$), we have
\begin{align}
-K & \leq-(\varphi^{\prime}(\left|z\right|)C_{H})^{p-2}\left|z\right|^{\alpha-2}\frac{C_{0}^{-2}\alpha\left|\alpha-1\right|}{4}C_{H}\nonumber \\
 & \leq-\frac{C_{0}^{-2}\alpha\left|\alpha-1\right|}{4}(\varphi^{\prime}(\left|z\right|)C_{H})^{p-2}\frac{\omega^{\alpha-2}(\left|z\right|)}{C_{H}^{\alpha-2}\varphi^{\prime}(\left|z\right|)^{\alpha-2}}C_{H}\nonumber \\
 & \leq-\frac{C_{0}^{-2}\alpha\left|\alpha-1\right|}{4}(\varphi^{\prime}(\left|z\right|)C_{H})^{p-\alpha}(\osc_{Q}u)^{\alpha-2}C_{H},\label{eq:ishii lions 11}
\end{align}
where $p-\alpha>0$. Therefore, since $\varphi^{\prime}(\left|z\right|)>\alpha2^{\alpha-1}$,
taking large enough $C_{H}$ in (\ref{eq:ishii lions 11}) yields
a contradiction. Thus $\Psi$ is non-positive, which
by the definition of $\varphi$ yields
\[
u(x_{0},t_{0})-u(y_{0},t_{0})\leq C_{H}\varphi(\left|x_{0}-y_{0}\right|)\leq C_{H}\left|x_{0}-y_{0}\right|^{\alpha}.
\]
Since $x_{0},y_{0}$ and $t_{0}$ were arbitrary, $u$ is H{\"o}lder continuous
in space in $Q_{1/2}$. Repeating the argument, we see that $u$ is
H{\"o}lder continuous in $Q_{1}$.
\end{proof}
We can now finish the proof of Theorem \ref{thm:lipschitz theorem}. In its proof, we apply Lemma \ref{lem:H=0000F6lder in space} with $\alpha = 1/2$, but any $\alpha \in (0,1)$ would suffice.
\begin{proof}[Proof of Theorem \ref{thm:lipschitz theorem}]
 By Lemma \ref{lem:H=0000F6lder in space}, the function $u$ is H{\"o}lder
continuous in particular in $Q_{1}$. Therefore, denoting by $\omega$
the optimal modulus of continuity of $u$ in $Q_{1}$ in space variable,
we have
\begin{equation}
\left|u(x,t)-u(y,t)\right|\leq\omega(\left|x-y\right|)\leq C_{H}\left|x-y\right|^{\alpha}\quad\text{for all }(x,t),(y,t)\in Q_{1},\label{eq:modulus form}
\end{equation}
where $\alpha:=1/2$ and $C_{H}$ depends
only on $N,p,m$, and the optimal modulus of continuity of $u$ in
$B_{2}\times(-2,0)$.

We define $\varphi:[0,2]\rightarrow[0,\infty)$ by $\varphi(s):=s-\kappa s^{\beta},$
where $\beta:=\frac{\alpha}{2}+1$ and $\kappa:=\beta^{-1}2^{-\beta-1}$.
Then
\begin{align*}
\varphi^{\prime}(s) & =1-2^{-\beta-1}s^{\beta-1}\quad\text{and}\quad\varphi^{\prime\prime}(s)=-2^{-\beta-1}(\beta-1)s^{\beta-2}
\end{align*}
so that, since $\beta\in(1,2)$, we have for $s\in(0,2)$
\begin{align*}
\left|\varphi^{\prime\prime}(s)\right|-\frac{\left|\varphi^{\prime}(s)\right|}{s} & =2^{-\beta-1}(\beta-1)s^{\beta-2}-\frac{1-2^{-\beta-1}s^{\beta-1}}{s}\\
 & =2^{-\beta-1}\beta s^{\beta-2}-s^{-1}\\
 & =s^{-1}(2^{-\beta-1}\beta s^{\beta-1}-1)\\
 & \leq s^{-1}(2^{-2}\beta-1)<0
\end{align*}
and thus the requirement (\ref{eq:Ishii-Lions lemma cnd}) of Lemma
\ref{lem:Ishii-Lions lemma} holds. We fix $K:=8\osc_{Q_{1}}u$ and
$(x_{0},t_{0}),(y_{0},t_{0})\in B_{1/2}\times(-1/2,1/2)$ and define
\[
\Psi(x,y,t):=u(x,t)-u(y,t)-L\varphi(\left|x-y\right|)-\frac{K}{2}\left|x-x_{0}\right|^{2}-\frac{K}{2}\left|y-y_{0}\right|^{2}-\frac{K}{2}\left|t-t_{0}\right|^{2}.
\]
It now suffices to prove that $\Psi$ is non-positive. Suppose on
the contrary that $\Psi$ has a positive maximum at $(\hat{x},\hat{y},\hat{t})\in\overline{B}_{1}\times\overline{B}_{1}\times[-1,1]$.
Then, assuming that $L\geq C(\alpha,\osc_{Q}u)$ it follows from Lemma
\ref{lem:Ishii-Lions lemma} that
\begin{align}
\label{eq:ishii-lions-Lip}
-K\leq & (L\varphi^{\prime}(\left|z\right|))^{p-2}\left(C_{0}^{-2}L\varphi^{\prime\prime}(\left|z\right|)+L\frac{\omega(\left|z\right|)\varphi^{\prime}(\left|z\right|)}{\left|z\right|}+\sqrt{K}\frac{\omega^{1/2}(\left|z\right|)}{\left|z\right|}+K\right)
\end{align}
for some $C_{0}$ that depends only on $N,p,m$ and $M$, and $z=\hat{x}-\hat{y}$,
for some $(\hat{x},\hat{t}),(\hat{y},\hat{t})\in Q_{1}$. We estimate
the right-hand side using (\ref{eq:modulus form}), definition of
$\beta$, formula of $\varphi^{\prime\prime}$ and the fact that $\varphi^{\prime}\in[3/4,1]$.
We obtain
\begin{align}
-K & \leq CL^{p-2}(-\kappa\beta(\beta-1)C_{0}^{-2}L\left|z\right|^{\beta-2}+C_{0}C_{H}L\left|z\right|^{\alpha-1}+\sqrt{K}C_{H}^{1/2}\left|z\right|^{\frac{\alpha}{2}-1}+K)\nonumber \\
 & =CL^{p-2}(-\kappa\beta(\beta-1)C_{0}^{-2}L\left|z\right|^{\frac{\alpha}{2}-1}+C_{0}C_{H}L\left|z\right|^{\alpha-1}+\sqrt{K}C_{H}^{1/2}\left|z\right|^{\frac{\alpha}{2}-1}+K)\nonumber \\
 & =CL^{p-2}\left|z\right|^{\frac{\alpha}{2}-1}(-C(\alpha)C_{0}^{-2}L+C_{0}C_{H}L\left|z\right|^{\frac{\alpha}{2}}+\sqrt{K}C_{H}^{1/2}+K),\label{eq:lipschitz 1}
\end{align}
where we also used that $\kappa\beta(\beta-1)=C(\alpha)$. By (\ref{eq:lemma derivative est})
we have
\begin{align*}
\left|z\right|^{\frac{\alpha}{2}}\leq & \left(\frac{\omega(\left|z\right|)}{\varphi^{\prime}(\left|z\right|)L}\right)^{\frac{\alpha}{2}}\leq C(\alpha,\osc_{Q_{1}}u)L^{-\frac{\alpha}{2}},
\end{align*}
and thus continuing the estimate (\ref{eq:lipschitz 1}) we get
\begin{align*}
-K & \leq CL^{p-2}\left|z\right|^{\frac{\alpha}{2}-1}(-C(\alpha)C_{0}^{-2}L+C_{0}C_{H}C(\alpha,\osc_{Q_{1}}u)L^{1-\frac{\alpha}{2}}+\sqrt{K}C_{H}^{1/2}+K).
\end{align*}
Observe that the negative term inside the parentheses has the highest
power of $L$. Therefore, assuming $L$ to be large enough depending
on $\alpha,C_{0},C_{H},K$ and $\osc_{Q_{1}}u$, we can absorb the
positive terms to the negative term. This way, we arrive at
\[
-K\leq-C(\alpha,C_{0})\left|z\right|^{\frac{\alpha}{2}-1}L^{p-1}.
\]
Since $\left|z\right|^{\frac{\alpha}{2}-1}>C$, a contradiction follows
by taking even larger $L$ if necessary.

Therefore, $\Psi$ is non-positive, which by the definition of $\varphi$
means that
\[
u(x_{0},t_{0})-u(y_{0},t_{0})\leq L\varphi(\left|x_{0}-y_{0}\right|)\leq L\left|x_{0}-y_{0}\right|.
\]
Since $x_{0},y_{0}$ and $t_{0}$ were arbitrary, this proves the
Lipschitz continuity of $u$ in $Q_{1/2}$.
\end{proof}

\section{Local H{\"o}lder continuity in time}
We use Lipschitz continuity in space to show that positive solutions to Trudinger's equation are locally H{\"o}lder continuous in time with exponent $1/2$. While H{\"o}lder continuity of weak solutions was already established in \cite{kuusisu12, kuusilsu12}, here we have an explicit exponent. The proof is based on the comparison principle and suitable barriers (see e.g.\ \cite[Lemma 9.1]{barlesBitonLey2002} or \cite[Lemma 3.1]{imbertJinSilvestre19} for similar arguments).

\begin{thm}\label{thm:holder continuity in time}
    Let $u$ be a uniformly continuous viscosity solution to (\ref{eq:trudinger}) in $B_4\times(-4,0)$. Suppose that $0< m \leq u \leq M$. Then there is $C_H >0$ such that
    \begin{equation}\label{eq:time holder}
        |u(x_0, t_0) - u(x_0, s_0)| \leq C_H |t_0 - s_0|^{1/2} \text{\quad for all } (x_0, t_0), (x_0, s_0) \in Q_1.
    \end{equation}
    The constant $C_H$ depends only on $N$, $p$, $m$, $M$ and the optimal modulus of continuity of $u$ in space in $B_4\times (-4,0)$.
\end{thm}
\begin{proof}
The idea of the proof is to construct an explicit barrier function. Using the spatial Lipschitz regularity result, we show that this barrier lies above the solution on the parabolic boundary of the cylinder (Step 1).  By choosing the constants in the barrier appropriately, we also ensure that it is a supersolution to Trudinger's equation (Step 2).   The comparison principle then implies that the barrier remains above the solution  in the whole cylinder and thus repeating the same argument from below, we obtain the desired estimate in time (Step 3).

  {\bf Step 1:} Let $t_0 \in (-1,0)$ and $s_0 \in (t_0,0]$ be arbitrary. Then it suffices to establish (\ref{eq:time holder}) when $x_0 = 0$ since the case $x_0 \in B_1$ is analogous. By Theorem \ref{thm:lipschitz theorem}, the function $u$ is $L$-Lipschitz in $Q_1$ in space for some $L>0$. We define 
    \[
    \varphi(x,t):=u(0,t_{0})+A+\Theta(t-t_{0})+K\left|x\right|^{\beta},
    \]
    where $\beta:=p/(p-1)$, $\Theta\geq0$ is chosen later,
    \[
    A:=(s_{0}-t_{0})^{1/2}\quad\text{and}\quad K:=\max\left(M,\beta^{-1} L^{\beta}A^{1-\beta}\right).
    \]
    Since $K\geq M$ we have
    $u\leq\varphi$ on $[t_{0},0]\times\partial B_{1}$. Since $u$ is
    $L$-Lipschitz in space variables, for all $x\in B_{1}$ we have that
    \begin{align}
    \varphi(x,t_{0})-u(x,t_{0}) & =u(0,t_{0})-u(x,t_{0})+A+K\left|x\right|^{\beta}\nonumber \\
    & \geq-L\left|x\right|+A+K\left|x\right|^{\beta}=:f(\left|x\right|),\label{eq:holder 1}
    \end{align}
    where we denoted $f(r):=A-Lr+Kr^{\beta}$. Observe that
    \[
    f^{\prime}(r)=-L+K\beta r^{\beta-1}=0\iff r=\left(\frac{L}{K\beta}\right)^{\frac{1}{\beta-1}}=:r_{0}
    \]
    so that
    \begin{align*}
    \min_{r\in[0,\infty)}f(r)=f(r_{0}) & =A-L\left(\frac{L}{K\beta}\right)^{\frac{1}{\beta-1}}+K\left(\frac{L}{K\beta}\right)^{\frac{\beta}{\beta-1}}\\
    & =A+K^{-\frac{1}{\beta-1}}L^{\frac{\beta}{\beta-1}}\left(-\left(\frac{1}{\beta}\right)^{\frac{1}{\beta-1}}+\left(\frac{1}{\beta}\right)^{\frac{\beta}{\beta-1}}\right)\\
    & \geq A-\left(\frac{1}{K}\frac{L^{\beta}}{\beta}\right)^{\frac{1}{\beta-1}}\geq0,
    \end{align*}
    where in the last inequality we used the definition of $K$. Thus by
    (\ref{eq:holder 1}) we have $u\leq\varphi$ on $B_{1}\times\left\{ t_{0}\right\} $.
    So far we have shown that for any $\Theta\geq0$
    \begin{equation}
    u\leq\varphi\quad\text{on }\partial_{\mathcal{P}}(B_{1}\times[t_{0},0]).\label{eq:holder 7}
    \end{equation}
    
{\bf Step 2:} Next, we select $\Theta$ so large that
    \begin{equation}
    \partial_{t}\varphi^{p-1}(x,t)\geq\Delta_{p}\varphi(x,t)\quad\text{for all }(x,t)\in((\mathbb{R}^{N}\setminus\left\{ 0\right\} )\times(t_{0},\infty))\cap\left\{ \varphi<M\right\} .\label{holder 6}
    \end{equation}
    To this end, fix such a point $(x,t)$. Observe that since $x\not=0$,
    we have
    \begin{align}
    \left|\Delta_{p}\varphi(x,t)\right| & \leq\left|D\varphi(x,t)\right|^{p-2}\left\Vert D^{2}\varphi(x,t)\right\Vert \nonumber \\
    & =\left|K\beta\left|x\right|^{\beta-1}\right|^{p-2}\left\Vert K\beta\left|x\right|^{\beta-2}I+K\beta(\beta-2)x\otimes x\left|x\right|^{\beta-4}\right\Vert \nonumber \\
    & \le C(N,p)K^{p-1}\left|x\right|^{(\beta-1)(p-1)-1}=C(N,p)K^{p-1}.\label{eq:holder 3}
    \end{align}
    where we used that $(\beta-1)(p-1)=1$. On the other hand, since by
    definition $\varphi>m$ and because $(x,t)\in\left\{ \varphi<M\right\} $,
    we have that
    \begin{align}
    (p-1)\varphi^{p-2}(x,t)\partial_{t}\varphi(x,t) & =(p-1)\Theta\varphi^{p-2}(x,t)\nonumber \\
    & >(p-1)\min(M^{p-2},m^{p-2})\Theta\nonumber \\
    & =C(p,m,M)\Theta.\label{eq:holder 4}
    \end{align}
    Taking $\Theta:=C(N,p,m,M)K^{p-1}$, the inequality (\ref{holder 6})
    holds by (\ref{eq:holder 3}) and (\ref{eq:holder 4}).

    We define 
    \[
    \overline{\varphi}(x,t):=\min(\varphi(x,t),M).
    \]
    We show that $\overline{\varphi}$ is a weak supersolution to Trudinger's
    equation in $B_{1}\times(t_{0},0)$. To this end, let $\phi\in C_{0}^{\infty}(B_{1}\times(t_{0},0))$
    be non-negative. Observe that $\varphi(x,t)<M$ if and only if $x\in B_{\rho(t)}$,
    where 
    \[
    \rho(t)=\left(\frac{M-u(t_{0},0)-A-\Theta(t-t_{0})}{K}\right)^{1/\beta}.
    \]
    Since $\varphi$ is $C^{1}$, the weak derivatives $\partial_{t}\overline{\varphi}$
    and $D\overline{\varphi}$ exist, and
    \[
    (\partial_{t}\overline{\varphi},D\overline{\varphi})=\begin{cases}
    (\partial_{t}\varphi,D\varphi), & \text{a.e.\ in }\left\{ \varphi<M\right\} ,\\
    (0,0), & \text{a.e.\ in }\left\{ \varphi\geq M\right\} .
    \end{cases}
    \]
    Hence
    \begin{align*}
    & \int_{B_{1}\times(t_{0},0)}-\overline{\varphi}^{p-1}\partial_{t}\phi+\left|D\overline{\varphi}\right|^{p-2}D\overline{\varphi}\cdot D\phi \d z \\
    & =\int_{t_{0}}^{0}\int_{B_{\rho(t)}}\phi\partial_{t}\varphi^{p-1}+\left|D\varphi\right|^{p-2}D\varphi\cdot D\phi\d x\d t\\
    & =\lim_{r\rightarrow0}\Big(\int_{t_{0}}^{0}\int_{B_{\rho(t)}\setminus B_{r}}\phi\partial_{t}\overline{\varphi}^{p-1}+\left|D\varphi\right|^{p-2}D\varphi\cdot D\phi\d x\d t\Big),
    \end{align*}
    where by the divergence theorem
    \begin{align*}
    & \int_{t_{0}}^{0}\int_{B_{\rho(t)}\setminus B_{r}}\left|D\varphi\right|^{p-2}D\varphi\cdot D\phi\d x\d t\\
    & =\int_{t_{0}}^{0}\int_{B_{\rho(t)}\setminus B_{r}}-\phi\div(\left|D\varphi\right|^{p-2}D\varphi)dx\d t\\
    & \ \ \ -\int_{t_{0}}^{0}\int_{\partial B_{r}}\phi\left|D\varphi\right|^{p-2}D\varphi\cdot\frac{x}{\left|x\right|}\d S(x)\d t+\int_{t_{0}}^{0}\int_{\partial B_{\rho(t)}}\phi\left|D\varphi\right|^{p-2}D\varphi\cdot\frac{x}{\left|x\right|}\d S(x)\d t\\
    & \geq\int_{t_{0}}^{0}\int_{B_{\rho(t)}\setminus B_{r}}-\phi\partial_{t}\varphi^{p-1}\d x\d t-\int_{t_{0}}^{0}\int_{\partial B_{r}}\phi\left|D\varphi\right|^{p-2}D\varphi\cdot\frac{x}{\left|x\right|}\d S(x)\d t.
    \end{align*}
    Here we used (\ref{eq:holder 7}) and that $D\varphi(x,t)\cdot x\geq0$
    for all $(x,t)\in\mathbb{R}^{N+1}$. Combining the last two displays,
    we see that $\overline{\varphi}$ is a weak supersolution in $B_{1}\times(t_{0},0)$. 

        {\bf Step 3:} By (\ref{eq:holder 7}) and since $u \leq M$, we have $u\leq\overline{\varphi}$
    on $\partial_{\mathcal{P}}(B_{1}\times(t_{0},0))$. Therefore, since $u$ is a weak solution by Theorem \ref{thm:equivalence}, it
    follows from comparison principle \cite[Corollary 1]{lindgrenl22} that $u\leq\overline{\varphi}$
    in $B_{1}\times[t_{0},0]$. In particular, we have that
    \begin{align*}
    u(0, s_{0})-u(0, t_{0}) & \leq\overline{\varphi}(0,s_0)-u(0, t_{0})\\
    & \leq A+C(N,p,m,M)(2M+\beta^{-1}A^{1-\beta}L^{\beta})^{p-1}(s_{0}-t_{0})\\
    & \leq A+C(N,p,m,M,L)(1+A^{(1-\beta)(p-1)})(s_{0}-t_{0}).
    \end{align*}
    Since $(1-\beta)(p-1)=-1$ and $A=(s_{0}-t_{0})^{1/2},$ we obtain
    \[
    u(0,s_0)-u(0, t_0)\leq C(N,p,m,M,L)(s_{0}-t_{0})^{1/2}.
    \]
    The lower bound can be derived similarly, but instead using the subsolution
    \[
    \underline{\varphi}(x,t):=\max(u(0, t_0)-A-\Theta(t-t_{0})-K\left|x\right|^{\beta},m). \qedhere
    \]

\end{proof}

\appendix
\section*{Appendix. Basic properties}

For the benefit of the reader, we prove the basic properties of inf-convolution in  Lemma \ref{lem:inf properties} even if they are rather standard. For the convenience of the reader we repeat the lemma here.
\newtheorem*{lem*}{Lemma}
\begin{lem*}
    \label{lem:inf properties-A}Assume that $u:\Xi\rightarrow\mathbb{R}$ is lower semicontinuous
and bounded. Suppose also that $\delta_{\varepsilon}\rightarrow0$
as $\varepsilon\rightarrow0$. Then $u_{\varepsilon}$ has the following
properties.
\begin{enumerate}[label=(\roman*)]
\item We have $u_{\varepsilon}\leq u$ in $\Xi$ and $u_{\varepsilon}\rightarrow u$
pointwise as $\varepsilon\rightarrow0$.
\item Denote $r(\varepsilon):=(q\varepsilon^{q-1}\osc_{\Xi}u)^{\frac{1}{q}}$,
$t(\varepsilon):=(2\delta_{\varepsilon}\osc_{\Xi}u)^{\frac{1}{2}}$ and set
\[
\Xi_{\varepsilon}:=\left\{ (x,t)\in\Xi:B_{r(\varepsilon)}(x)\times(t-t(\varepsilon),t+t(\varepsilon))\Subset\Xi\right\} .
\]
Then, for any $(x,t)\in\Xi_{\varepsilon}$ there exists $(x_{\varepsilon},t_{\varepsilon})\in\overline{B}_{r(\varepsilon)}(x)\times[t-t(\varepsilon),t+t(\varepsilon)]$
such that
\[
u_{\varepsilon}(x,t)=u(x_{\varepsilon},t_{\varepsilon})+\frac{\left|x-x_{\varepsilon}\right|^{q}}{q\varepsilon^{q-1}}+\frac{\left|t-t_{\varepsilon}\right|^{2}}{2\delta_{\varepsilon}}.
\]
\item The function $u_{\varepsilon}$ is semi-concave in $\Xi_{\varepsilon}$.
In particular, the function $u_{\varepsilon}(x,t)-(C\left|x\right|^{2}+t^{2}/\delta_{\varepsilon})$
is concave in $\Xi_{\varepsilon}$, where $C=(q-1)r(\varepsilon)^{q-2}/\varepsilon^{q-1}$.
\item Suppose that $u_{\varepsilon}$ is twice differentiable in space and
time at $(x,t)\in\Xi_{\varepsilon}$. Then 
\begin{align*}
\partial_{t}u_{\varepsilon}(x,t) & =\frac{t-t_{\varepsilon}}{\delta_{\varepsilon}},\\
Du_{\varepsilon}(x,t) & =(x-x_{\varepsilon})\frac{\left|x-x_{\varepsilon}\right|^{q-2}}{\varepsilon^{q-1}},\\
D^{2}u_{\varepsilon}(x,t) & \leq(q-1)\frac{\left|x-x_{\varepsilon}\right|^{q-2}}{\varepsilon^{q-1}}I.
\end{align*}
\end{enumerate}
\end{lem*}

\begin{proof}[Proof of Lemma \ref{lem:inf properties}] We consider the different properties given in the lemma separately.

\textit{Proof of (ii)}. Let $(x,t)\in\Xi_{\varepsilon}$. Let $(y,s)\in\Xi$
be such that $y\in\mathbb{R}^{N}\setminus\overline{B}_{r(\varepsilon)}(x)$
or $s\in\mathbb{R}^{N}\setminus[t-t(\varepsilon),t+t(\varepsilon)]$.
Then by the definition of $r(\varepsilon)$ and $t(\varepsilon)$, we
have that
\begin{align*}
\frac{\left|y-x\right|^{q}}{q\varepsilon^{q-1}}+\frac{\left|s-t\right|^{2}}{2\delta_{\varepsilon}} & >\osc_{\Xi}u
\end{align*}
so that 
\begin{align*}
 u(y,s)+\frac{\left|y-x\right|^{q}}{q\varepsilon^{q-1}}+\frac{\left|s-t\right|^{2}}{2\delta_{\varepsilon}}>u(y,s)+u(x,t)-u(x,t)+\osc_{\Xi}u\geq u(x,t)\geq u_{\varepsilon}(x,t).   
\end{align*}
Hence by lower semi-continuity
\begin{align*}
u_{\varepsilon}(x,t) & =\inf_{(y,s)\in\overline{B}_{r(\varepsilon)}(x)\times[t-t(\varepsilon),t+t(\varepsilon)]}\left\{ u(y,s)+\frac{\left|y-x\right|^{q}}{q\varepsilon^{q-1}}+\frac{\left|s-t\right|^{2}}{2\delta_{\varepsilon}}\right\} \\
 & =u(x_{\varepsilon},t_{\varepsilon})+\frac{\left|x-x_{\varepsilon}\right|^{q}}{q\varepsilon^{q-1}}+\frac{\left|t-t_{\varepsilon}\right|^{2}}{2\delta_{\varepsilon}}
\end{align*}
for some $(x_{\varepsilon},t_{\varepsilon})\in\overline{B}_{r(\varepsilon)}\times[t-t(\varepsilon),t+t(\varepsilon)]$.

\textit{Proof of (i)}. Let $(x,t)\in\Xi$ and let $\varepsilon>0$ be so small
that $(x,t)\in\Xi_{\varepsilon}$. From \textit{(ii)} and boundedness of $u$,
it is clear that $(x_{\varepsilon},t_{\varepsilon})\rightarrow(x,t)$.
Therefore, by lower semi-continuity
\[
\limsup_{\varepsilon\rightarrow0}\frac{\left|x-x_{\varepsilon}\right|^{q}}{q\varepsilon^{q-1}}+\frac{\left|t-t_{\varepsilon}\right|^{2}}{2\delta_{\varepsilon}}=\limsup_{\varepsilon\rightarrow0}(u_{\varepsilon}(x,t)-u(x_{\varepsilon},t_{\varepsilon}))\leq\limsup_{\varepsilon\rightarrow0}(u(x,t)-u(x_{\varepsilon},t_{\varepsilon}))\leq0
\]
and so $u_{\varepsilon}(x,t)\rightarrow u(x,t)$.

\textit{Proof of (iii)}. For all $(y,s)\in\Xi$, define the
function $\varphi_{(y,s)}:\mathbb{R}^{N+1}\rightarrow\mathbb{R}$
by
\[
\varphi_{(y,s)}(x,t):=u(y,s)+\frac{\left|y-x\right|^{q}}{q\varepsilon^{q-1}}+\frac{\left|s-t\right|^{2}}{2\delta_{\varepsilon}}.
\]
Observe that for any $h>0$, whenever $\left|y-x\right|\le r(\varepsilon)+h$,
we have that
\begin{align*}
D^{2}\varphi_{(y,s)}(x,t) & \leq\frac{(q-1)\left|y-x\right|^{q-2}}{\varepsilon^{q-1}}\leq\frac{(q-1)(r(\varepsilon)+h)^{q-2}}{\varepsilon^{q-1}},\\
\partial_{tt}\varphi_{(y,s)}(x,t) & \le\frac{1}{\delta_{\varepsilon}}.
\end{align*}
It follows that for $C_{h}:=(q-1)(r(\varepsilon)+h)^{q-2}/\varepsilon^{q-1}$
the function $(x,t)\mapsto\varphi_{(y,s)}(x,t)-(C_h\left|x\right|^{2}+t^{2}/\delta_{\varepsilon})$
is concave in the set $\left\{ (x,t)\in\mathbb{R}^{N+1}:\left|x-y\right|\leq r(\varepsilon)+h\right\} $.
Now, if $(z, \tau)\in\Xi_{\varepsilon}$, then for small $h>0$
we have $B_{h}(z)\times(\tau-h,\tau+h)\subset\Xi_{\varepsilon}$.
Therefore, by property \textit{(ii)} and the definition of inf-convolution, we have
for all $(x,t)\in B_{h}(z)\times(\tau-h,\tau+h)$ that
\begin{align*}
u_{\varepsilon}(x,t)-(C_h\left|x\right|^{2}+t^{2}/\delta_{\varepsilon}) & =\inf_{(y,s)\in\overline{B}_{r(\varepsilon)}(x)\times[t-t(\varepsilon),t+t(\varepsilon)]}\left\{ \varphi_{(y,s)}(x,t)\right\} -C_{h}(\left|x\right|^{2}+\left|t\right|^{2})\\
 & =\inf_{(y,s)\in\Xi\cap\overline{B}_{r(\varepsilon)+h}(z)\times[\tau-t(\varepsilon)-h,\tau+t(\varepsilon)+h]}\left\{ \varphi_{(y,s)}(x,t)\right\} -C_{h}(\left|x\right|^{2}+\left|t\right|^{2})\\
 & =\inf_{(y,s)\in\Xi\cap\overline{B}_{r(\varepsilon)+h}(z)\times[\tau-t(\varepsilon)-h,\tau+t(\varepsilon)+h]}\left\{ \varphi_{(y,s)}(x,t)-C_{h}(\left|x\right|^{2}+\left|t\right|^{2})\right\} .
\end{align*}
It follows that as the
infimum of concave functions $u_{\varepsilon}(x,t)-(C_{h}\left|x\right|^{2}+t^{2}/\delta_{\varepsilon})$
is concave in $B_{h}(z)\times(\tau-h,\tau+h)$, and thus in $\Xi_{\varepsilon}$, since
$(z,\tau)\in\Xi_{\varepsilon}$ was arbitrary. Taking $h\rightarrow0$,
we see that also $u_{\varepsilon}(x,t)-(C_{0}\left|x\right|^{2}+t^{2}/\delta_{\varepsilon})$
is concave.

\textit{Proof of (iv)}. If $u_{\varepsilon}$ is twice differentiable at $(x,t)\in\Xi_{\varepsilon}$,
then there is a $C^{2}$-function $\varphi$ that touches $u_\varepsilon$ from
below at $(x,t)$ and $\partial_{t}\varphi(x,t)=\partial_{t}u_\varepsilon(x,t)$,
$D\varphi(x,t)=Du_\varepsilon(x,t)$ as well as $D^{2}\varphi(x,t)=D^{2}u_\varepsilon(x,t)$. By \textit{(ii)} and the definition of inf-convolution, we have
\begin{align*}
\varphi(x,t) & = u_{\varepsilon}(x,t)=u(x_{\varepsilon},t_{\varepsilon})+\frac{\left|x-x_{\varepsilon}\right|^{q}}{q\varepsilon^{q-1}}+\frac{\left|t-t_{\varepsilon}\right|^{2}}{2\delta_{\varepsilon}},\\
\varphi(z,\tau) & \leq u_{\varepsilon}(z,\tau)\leq u(y,s)+\frac{\left|y-z\right|^{q}}{q\varepsilon^{q-1}}+\frac{\left|\tau-s\right|^{2}}{2\delta_{\varepsilon}}\quad\text{for all }(z,\tau),(y,s)\in\Xi.
\end{align*}
But this means that the function $\psi(z,\tau):=\varphi(z,\tau)-\frac{\left|z-x_{\varepsilon}\right|^{q}}{q\varepsilon^{q-1}}-\frac{\left|\tau-t_{\varepsilon}\right|^{2}}{2\delta_{\varepsilon}}$
has a maximum at $(z,\tau)=(x,t)$. Indeed,
the first display above implies that $\psi(x,t)=u(x_{\varepsilon},t_{\varepsilon})$,
and selecting $(y,s)=(x_{\varepsilon},t_{\varepsilon})$ in the second one
we get that $\psi(z,\tau)\leq u(x_{\varepsilon},t_{\varepsilon})$ for all $(z, \tau)\in\Xi$.
The claim then follows by direct computation.
\end{proof}

Next we restate and prove Lemma~\ref{lem: elem ineq} for the convenience of the reader.
\begin{lem*}[Elementary inequality]
 Let $p>1$. Suppose that $a\geq m>0$. Then
for all $s\in(0,1)$ we have
\[
\left|(a+s)^{\frac{2-p}{2}}-a^{\frac{2-p}{2}}\right|\leq C(p,m)s.
\]
\end{lem*}
\begin{proof}[Proof of Lemma~\ref{lem: elem ineq}.]
Clearly
\begin{align*}
\left|(a+s)^{\frac{2-p}{2}}-a^{\frac{2-p}{2}}\right|=\left|\int_{a}^{a+s}\frac{d}{dh}h^{\frac{2-p}{2}}\d h\right| & =\left|\int_{a}^{a+s}\frac{p-2}{2}h^{-\frac{p-2}{2}-1}\d h\right|\\
 & \leq\frac{\left|p-2\right|}{2}\max(a^{-\frac{p-2}{2}-1},(a+s)^{-\frac{p-2}{2}-1})s\\
 & =C(p,m)s,\qedhere
\end{align*}
where one can check that $C(p,m)=\frac{|p-2|}{2}m^{-p/2}$.
\end{proof}

Data Availability Statement: This paper has no associated data.

Conflict of
interest: On behalf of all authors, the corresponding author states that there is no conflict of
interest.

\bibliographystyle{alpha}
\bibliography{ref}

\end{document}